 \documentclass[a4paper,11pt]{article}
\usepackage{pdfsync}

 \usepackage[plainpages=false]{hyperref}
 \usepackage{amsfonts,amsmath,amssymb,amsthm}
 \usepackage{latexsym,lscape,rawfonts,mathrsfs}
 \usepackage{mathtools}
 \usepackage[utf8]{inputenc}
 \usepackage[overload]{empheq}
 \usepackage{cases}

 \usepackage[dvips]{color}
 \usepackage{multicol}

 \usepackage[all]{xy}
 \usepackage{makeidx}
 \usepackage{graphicx,psfrag}
 \usepackage{pstool}
 \usepackage{float}

 \usepackage{array,tabularx}

 \usepackage{setspace}

 \usepackage[titletoc]{appendix}

\usepackage{txfonts}

\newtheorem{prop}{Proposition}[section]

\newtheorem{proposition}[prop]{Proposition}

\newtheorem{theorem}[prop]{Theorem}

\newtheorem{lemma}[prop]{Lemma}

\newtheorem{corollary}[prop]{Corollary}

\newtheorem{remark}[prop]{Remark}

\newtheorem{definition}[prop]{Definition}

\numberwithin{equation}{section}

\title{Rigidity and $\varepsilon$-regularity theorems of Ricci shrinkers}
 \author{Jie Wang and Youde Wang}
 \author{Jie Wang \; and\; Youde Wang\footnote{corresponding author}}
\date{}
\begin{document}	
	\maketitle
\begin{abstract}
 In this paper, we study the rigidity and $\varepsilon$-regularity theorems of Ricci shrinkers. First we prove the rigidity of the asymptotic volume ratio and local volume around a base point of a non-compact Ricci shrinker. Next we obtain some $\varepsilon$-regularity theorems of local entropy and curvature, which improve the previous corresponding results essentially and use them to study the structure of Ricci shrinkers at infinity. Especially, if the curvature of a non-compact Ricci shrinker satisfies some natural integral conditions, then it is asymptotic to a cone.
\end{abstract}
\\
\tableofcontents

\section{Introduction}
Let $(M^n,g)$ be a complete Riemannian n-manifold and $Rm,Ric,R$ be the Riemannian curvature, Ricci curvature and scalar curvature respectively. $(M^n,g)$ is called a gradient Ricci soliton if there exists a smooth function $f$ and scalar $\lambda$ such that
$$Ric+Hessf=\lambda g.$$
A Ricci soliton is called shrinking, steady or expanding,  if $\lambda>0,\lambda=0$ or $\lambda<0$ respectively. By scaling the metric $g$, we may assume $\lambda\in \left\lbrace \frac{1}{2},0,-\frac{1}{2}\right\rbrace$, and in these situations, we call the solitons are normalized. By the above soliton equation, Ricci solitons are natural generalizations of Einstein manifolds. As we know, Ricci solitons are crucial in the singularity analysis of Ricci flows. For example, by \cite{EMT}, the blowups around a type-\uppercase\expandafter{\romannumeral 1} singularity point always converge to nontrivial gradient Ricci shrinking solitons. 

In this paper, we focus on the case of shrinking. A Ricci shrinker is a triple $(M^n, g,f)$ of smooth manifold $M^n$, Riemannian metric $g$ and a smooth function $f$ satisfying
\begin{equation}\label{1.1}
	Ric+Hessf=\frac{1}{2}g.
\end{equation}
Tracing (\ref{1.1}) gives
\begin{equation}\label{1.2}
	R+\Delta f=\frac{n}{2}.
\end{equation} 
By a normalization of $f$, we may always assume
\begin{align}
	R+\left|\nabla f\right|^2&=f,\label{1.3}\\
	\int(4\pi)^{-\frac{n}{2}}e^{-f}dV&=e^{\boldsymbol{\mu}},\label{1.4}
\end{align}
where $\boldsymbol{\mu}\leq0$ is the functional of Perelman(cf.\cite{LW}). Due to \cite{LW}, if $(M^n, g,f)$ is non-flat, then there exists a small $\varepsilon=\varepsilon(n)>0$ such that $\boldsymbol{\mu}\leq -\varepsilon$. Also, it's well-known that $R>0$, or $R\equiv0$ and the shrinker is flat. For convenience, we also denote a Ricci shrinker by $(M^n, g,f,p)$, where $p$ is any minimum point of $f$(cf. Lemma \ref{l2.1}). Moreover, we say $(M^n,g,f,p)\in\mathcal{M}(\boldsymbol{\mu}_0)$ if $\boldsymbol{\mu}(M^n,g)\geq\boldsymbol{\mu}_0$ for some uniform constant $\boldsymbol{\mu}_0\leq0$.

In the following, we use $B_r(x)$ or $B(x,r)$, and $\left| B_r(x)\right|$ or  $\left| B(x,r)\right|$ to denote the geodesic ball centered at $x$ with radius $r$ and its volume, respectively. If the following limit exists, the asymptotic volume ratio $\textbf{AVR}$ of a complete non-compact Riemannian manifold $(M^n, g)$ is defined by
$$\textbf{A}=\textbf{AVR}(M^n,g)=\lim_{r\rightarrow+\infty}\frac{\left| B_r(x)\right| }{\omega_nr^n},$$
 where $\omega_n$ is the volume of the unit Euclidean n-ball. Moreover, it's easy to know $\textbf{AVR}(M^n, g)$ is independent of the choice of $x$. If $Ric\geq0$, by classical Bishop-Gromov volume comparison theorem, we have $\textbf{AVR}(M^n,g)\leq 1$. If $\textbf{AVR}(M^n,g)$ exists and $>0$, we say $(M^n,g)$ has Euclidean volume growth. Asymptotic volume ratio of a manifold is an important geometric quantity which may reflect the properties at infinity. For instance, Anderson's gap theorem \cite[Lemma 3.1]{A2} proved that if $Ric\equiv0$ and $\textbf{AVR}>1-\epsilon$ for some $\epsilon=\epsilon(n)$, then $(M^n,g)$ must be isometric to the standard Euclidean space $(\mathbb{R}^n,g_E)$. Non-compact Ricci shrinkers shares many similar properties with manifolds possessing $Ric\geq 0$, such as their volume grows at least linearly(cf. \cite[Theorem 6.1]{MW1}, see also \cite[Proposition 6]{LW}). By \cite{CLY2}, the \textbf{AVR} of any non-compact Ricci shrinker exists and $\leq C(n)$. Now we confirm that Anderson's gap theorem is valid on Ricci shrinkers.
\begin{theorem}\label{t1.1*}
	For any non-compact Ricci shrinker $(M^n,g,f,p)$, there holds
	\begin{equation*}
		\boldsymbol{A}\leq e^{\boldsymbol{\mu}}(\leq1).
	\end{equation*}
	The equality holds if and only if $(M^n,g,f,p)$ is the standard Gaussian soliton $\left( \mathbb{R}^n, dx^2, \frac{1}{4}\left|x\right|^2\right)$, otherwise $\boldsymbol{A}\leq1-\varepsilon$ for some small $\varepsilon=\varepsilon(n)>0$.
\end{theorem}
On the basis of the above gap theorem at infinity, we prove the following local gap theorem which states that if the volume of some geodesic ball $B_r(p)$ is close to that of a standard Euclidean ball of the same radius, then the Ricci shrinker must be the standard Gaussian soliton.
\begin{theorem}\label{t1.2*}
	For any non-compact Ricci shrinker $(M^n,g,f,p)$, there exists an $\varepsilon=\varepsilon(n)>0$ such that if
	$$\omega_n^{-1}\varepsilon^n\left|B_{\frac{1}{\varepsilon}}(p)\right|\geq 1-\varepsilon,$$
	then it must be flat.
\end{theorem}

In light of the soliton equations (\ref{1.2}) and (\ref{1.3}), the behaviors of scalar curvature affects the geometric properties of Ricci shrinkers significantly. It is no exaggeration to say, the study of scalar curvature run through the entire paper. Especially, we are devoted to discuss the integral of scalar curvature, and then study the geometric structure at infinity via it. First of all, we concentrate on the global integral of scalar curvature.
\begin{theorem}[\textbf{non-integrability}]\label{t1.3*}
	For any non-compact and non-flat Ricci shrinker $(M^n,g,f,p)$, we have
	\begin{equation*}
		\int_{M}R^k=+\infty\quad if \quad 1\leq k\leq \frac{n}{2},
	\end{equation*}
	and there exists a constant $C=C(n,\boldsymbol{\mu},k)>0$ such that
	\begin{equation*}
		\int_{M}R^k\geq C\quad if \quad k>\frac{n}{2}.
	\end{equation*}
	Furthermore, the critical index $k=\frac{n}{2}$ is sharp.
\end{theorem}

\begin{theorem}[\textbf{Gap Theorem}]\label{t1.4*}
	For any compact or non-compact Ricci shrinker $(M^n,g,f,p)\in\mathcal{M}(\boldsymbol{\mu}_0)$ and $k\geq 1$, there exists a constant $\delta=\delta(n,\boldsymbol{\mu}_0,k)>0$ such that if
	\begin{equation*}
		\int_{M}R^k\leq \delta,
	\end{equation*}
	then it must be flat.
\end{theorem}

Next, we turn to the local properties of Ricci shrinkers. By B. Wang \cite{Wb,WB2}, we know Perelman's local entropies are essential tools for analyzing the local geometry of manifolds, and later, Li-Wang \cite{LW} studied the local properties of Ricci shrinkers such as no-local-collapsing and pseudo-locality by virtue of the theory established in \cite{Wb,WB2}. Especially, the $\boldsymbol{\nu}$ entropy is related closely to isoperimetric inequality, volume ratio and so on, for more information, we refer to \cite{Wb,WB2}. Utilizing the two-sided pseudo-locality theorem of Li-Wang's recent work \cite{LW2} and blowup arguments, we prove an $\varepsilon$-regularity theorem of the local $\boldsymbol{\nu}$ entropy. 
\begin{theorem}\label{t1.5*}
	For any $x\in (M^n,g,f,p)$ and $r>0$, there exists a constant $\varepsilon=\varepsilon(n)>0$ such that if 
	\begin{equation*}
		\boldsymbol{\nu}\left(B\left(x,r\right) ,r^2\right)\geq -\varepsilon,
	\end{equation*}
	then
	\begin{equation*}
		\left|Rm\right|(x)\leq \varepsilon^{-1}r^{-2}.
	\end{equation*}
\end{theorem}

As we know, due to Anderson's classical work \cite{A2}, Einstein manifolds have nice regularities with respect to the integral of curvature. Later, Huang \cite{HS} and Ge-Jiang \cite{GJ} extended Anderson's theory to 4 dimensional Ricci shrinkers along different ways respectively. However, both of their methods depended on the Sobolev inequalities, which are quite rough on Ricci shrinkers for the lack of curvature bounds. Very recently, Li-Wang \cite{LW3} made a breakthrough and proved that any 4 dimensional k\"{a}hler Ricci shrinker has bounded curvature.
Now by the same arguments as to get Theorem \ref{t1.5*} and the weak-compactness of Ricci shrinkers \cite[Theorem 3.4]{LW3}, we improve the previous results of $\varepsilon$-regularity essentially. Furthermore, there are several important applications of our $\varepsilon$-regularity theorem, such as controlling the local $\boldsymbol{\nu}$ entropy, see section \ref{sec4} for details. 
\begin{theorem}\label{t1.6*}
	For any $x\in(M^4,g,f,p)$ and $r\in\left(0,1\right]$, there exists a  constant $\varepsilon=\varepsilon(\boldsymbol{\mu})>0$ such that if
\begin{equation*}
	\int_{B\left(x,2r\right)}R^{2}dV\leq\varepsilon,\quad f(x)\geq \varepsilon^{-1}r^{-2},
\end{equation*}
then
\begin{equation*}
	\left|Rm\right|(x)\leq \varepsilon^{-1}r^{-2}.
\end{equation*}
\end{theorem}

For any non-compact $(M^4,g,f,p)$, by Theorem $\ref{t1.3*}$, we know $R^2$ is non-integrable on the whole manifold. However, we also point out that if $q>2$, then $R^q$ may be integrable. Now by the $\varepsilon$-regularity of Theorem \ref{t1.6*}, we can go a step further and prove that if $R^q$ is integrable on $(M^4,g,f,p)$, then the Ricci shrinker is $C^k$ is asymptotic to a cone for all $k$.
\begin{theorem}\label{t1.8*}
	For any Ricci shrinker $(M^4,g,f,p)$, if for some $q>2$,
	\begin{equation*}
		\int_{M^4}R^qdV<+\infty,
	\end{equation*}
	then it is $C^k$ asymptotic to a cone for all $k$.
\end{theorem}

For general dimension $n\geq 4$, under stronger condition, we still have
\begin{theorem}\label{t1.9*}
	For any $x\in(M^n,g,f,p)$ and $r>0$, there exists a constant $\varepsilon=\varepsilon\left(n, \boldsymbol{\mu}\right) >0$ such that if
	\begin{equation*}
		\int_{B\left(x,r\right)}\left|Rm\right|^{\frac{n}{2}}dV\leq\varepsilon,
	\end{equation*}
	then
	\begin{equation*}
		\left|Rm\right|(x)\leq \varepsilon^{-1}r^{-2}.
	\end{equation*}
Moreover, if for some $q>\frac{n}{2}$,
\begin{equation*}
	\int_{M^n}\left|Rm\right|^qdV<+\infty,
\end{equation*}
then it is $C^k$ asymptotic to a cone for all $k$.
\end{theorem}

\section{Asymptotic volume ratio}
First we recall 
\begin{lemma}[\cite{CLY2}]\label{l1.1}
	For any Ricci shrinker $(M^n,g,f,p)$, $\textbf{AVR}(M^n,g)$ always exists and is bounded by a constant depending only on n, i.e. $\textbf{AVR}(M^n,g)\leq C(n)$. Moreover, $\textbf{AVR}(M^n,g)>0$ if and only if $\int_{n+2}^{+\infty}\frac{\overline{X}(c)}{c\overline{V}(c)}dc<+\infty$ where $$\overline{V}(c)=\left|\overline{D}(c)\right|=\left|\left\lbrace x\in M: f<c\right\rbrace \right|  \quad and \quad \overline{X}(c)=\int_{\overline{D}(c)}RdV.$$
\end{lemma}

As we know, Cao-Zhou \cite{CZ} first obtained the $C^0$ estimates of the potential function $f$, and later, Haslhofer-M\"{u}ller \cite{HM} improved their result as following.
\begin{lemma}[cf. Haslhofer-M\"{u}ller \cite{HM}]\label{l2.1}
	Let $(M^n, g, f, p)$ be a Ricci shrinker, then there exists a point $p\in M$ where $f$ attains its infimum and $f$ satisfies the estimates
	\begin{equation}\label{2.1}
		\frac{1}{4}\left( d(x,p)-5n\right)^2_+\leq f(x)\leq \frac{1}{4}\left( d(x,p)+\sqrt{2n}\right)^2
	\end{equation}
	for all $x\in M$, where $d(x,p)$ is the distance function and $a_+\stackrel{\triangle}{=}max\left\lbrace a,0\right\rbrace$.
\end{lemma}

\begin{definition}
	A Ricci shrinker $(M^n,g,f,p)$ has Euclidean volume growth if $\textbf{AVR}(M^n,g)>0$. For simplicity, we always denote $\textbf{AVR}(M^n,g)$ by $\boldsymbol{A}$.
\end{definition}

As usual, we define
$$\rho(x)=2\sqrt{f(x)},$$
and denote
$$D(r)=\left\lbrace x\in M^n:\rho(x)<r\right\rbrace,\hspace{0.5em} V(r)=\int_{D(r)}dV\hspace{0.5em} and\hspace{0.5em} X(r)=\int_{D(r)}RdV.$$
By \cite{CZ}, 
\begin{align}
nV(r)-rV'(r)&=2X(r)-\frac{4}{r}X'(r),\label{1.5}\\
\frac{X}{V}&\leq\frac{n}{2}.\label{1.6}
\end{align} 
Due to the point-wise estimates of the potential function $f$ of Lemma \ref{l2.1}, $V(r)$ and $\left|B_r(p)\right|$ are almost the same if $r$ is large enough, thus
$$\textbf{AVR}(M^n,g)=\lim_{r\rightarrow+\infty}\frac{\left| B_r(p)\right| }{\omega_nr^n}=\lim_{r\rightarrow+\infty}\frac{V(r)}{\omega_nr^n}.$$

\begin{lemma}\label{l1.3}
For any Ricci shrinker $(M^n,g,f,p)$, if $\textbf{AVR}(M^n,g)>0$, then
$$\frac{X}{V}\longrightarrow0, \hspace{0.5em}\frac{X'}{Vr}\longrightarrow0\hspace{0.5em}as\hspace{0.5em}r\longrightarrow+\infty.$$
\end{lemma}
\begin{proof}
Let $r\geq\sqrt{2(n+2)}$ and set
$$P(r)=\frac{V}{r^n}-\frac{4X}{r^{n+2}}.$$
By (\ref{1.5}), it's easy to compute
\begin{equation}\label{1.7}
P'(r)=\frac{2X}{r^{n+1}}\left\lbrace\frac{2(n+2)}{r^2}-1\right\rbrace\leq0.
\end{equation}
In view of Lemma \ref{l1.1} and (\ref{1.6}), we have
$$\lim\limits_{r\longrightarrow+\infty}P(r)=\lim\limits_{r\longrightarrow+\infty}\frac{V(r)}{r^n}=\boldsymbol{A}\omega_n,$$
hence for $r\geq\sqrt{2(n+2)}$, 
\begin{align*}
\boldsymbol{A}\omega_n-P(r)&=\int_{r}^{+\infty}\frac{2X(s)}{s^{n+1}}\left\lbrace\frac{2(n+2)}{s^2}-1\right\rbrace ds\\
&\leq X(r)\int_{r}^{+\infty}\frac{2}{s^{n+1}}\left\lbrace\frac{2(n+2)}{s^2}-1\right\rbrace ds\\
&=X(r)\left\lbrace\frac{4}{r^{n+2}}-\frac{2}{nr^n}\right\rbrace,
\end{align*}
and consequently,
\begin{equation}\label{1.8}
\frac{X(r)}{r^n}\leq\frac{n}{2}\left\lbrace\frac{V(r)}{r^n}-\boldsymbol{A}\omega_n\right\rbrace\longrightarrow0. 
\end{equation}
On the other hand, since 
$$X(r+1)-X(r)\leq X(r+1)\leq\frac{n}{2}V(r+1),$$
there exists a $r_0\in\left[r,r+1\right]$ depending continuously on $r$ such that
$$X(r+1)-X(r)=X'(r_0)\leq\frac{n}{2}V(r+1),$$
therefore for all $r$ large enough, there holds
$$X'(r)\leq \frac{n}{2}V(r+1)\leq C(n)(r+1)^{n}.$$
The above inequality immediately implies
$$\frac{X'}{Vr}\longrightarrow0\hspace{0.5em}as\hspace{0.5em}r\longrightarrow+\infty.$$ 
\end{proof}

\begin{corollary}\label{c1.5}
If $\boldsymbol{A}=\textbf{AVR}(M^n,g)>0$, there exists a constant $C=C(n,\boldsymbol{A})$ such that 
for all $r>0$, 
$$\left|B_r(p)\right|\geq Cr^n.$$
\end{corollary}
\begin{proof}
By the proof of Lemma \ref{l1.3}, for any $r\geq\sqrt{2(n+2)}$,
$$P(r)=\frac{V}{r^n}-\frac{4X}{r^{n+2}}\geq \lim\limits_{r\longrightarrow+\infty}\frac{V(r)}{r^n}=\boldsymbol{A}\omega_n,$$
then $$V(r)\geq \boldsymbol{A}\omega_nr^n.$$
Hence by Lemma \ref{l1.1}, 
$$\left|B_{5n+r}(p)\right|\geq V(r)\geq \boldsymbol{A}\omega_nr^n= \boldsymbol{A}\omega_n\left( \frac{r}{5n+r}\right)^n(5n+r)^n\geq\boldsymbol{A}\omega_n\left( \frac{\sqrt{2(n+2)}}{5n+\sqrt{2(n+2)}}\right)^n(5n+r)^n.$$
If $r\leq\sqrt{2(n+2)}+5n$, then $R\leq f\leq\frac{\left(\sqrt{2(n+2)}+5n\right)^2}{4}$. By the no-local collapsing estimate of \cite[Theorem 23]{LW}, 
$$\left|B_r(p)\right|\geq\frac{c(n)e^{\boldsymbol{\mu}}r^n}{\left(1+\frac{\left(\sqrt{2(n+2)}+5n\right)^4}{4}\right)^{\frac{n}{2}}}.$$
Moreover, by \cite[Lemma 2]{LW}, $e^{\boldsymbol{\mu}}\geq c(n)\boldsymbol{A}$, and then the proof is complete.
\end{proof}

Next we study the well-defined integral
$$h(\tau)=\int_{M^n} udV,\hspace{0.5em}u=(4\pi\tau)^{-\frac{n}{2}}e^{-\frac{f}{\tau}},\hspace{0.5em}\tau>0,$$
which is similar to the reduced volume under a static metric(cf. \cite[chapter 8.1]{CCG}). Here we remark that the validity of $h(\tau)$ is ensured by the quadratic growth of $f$.
\begin{proposition}\label{p1.4}
$h(\tau)$ is increasing for $\tau<1$ and decreasing for $\tau>1$. If $\boldsymbol{A}>0$, there holds
\begin{equation}\label{1.9}
\lim\limits_{\tau\longrightarrow+\infty}h(\tau)=\boldsymbol{A}.
\end{equation}
\end{proposition}
\begin{proof}
First we have
\begin{equation}\label{1.10}
\frac{d}{d\tau}h=(4\pi)^{-\frac{n}{2}}\tau^{-\frac{n}{2}-1}\int_{M^n}\left(\frac{f}{\tau}-\frac{n}{2}\right)e^{-\frac{f}{\tau}}dV.
\end{equation}
By (\ref{1.2}) and (\ref{1.3}),
\begin{equation}\label{1.11}
\Delta f+f-\left|\nabla f\right|^2=\frac{n}{2}.
\end{equation}
Note that in the weighted space $\left( M^n,g,e^{-\frac{f}{\tau}}dV\right)$, 
$$\int_{M^n}\Delta_{\frac{f}{\tau}}fe^{-\frac{f}{\tau}}dV=0,$$
where $\Delta_f$ is defined as $\Delta_f=\Delta f-\left\langle\nabla f, \cdot \right\rangle$(cf.\cite{HM}), then we conclude
\begin{equation}\label{1.12}
\int_{M^n}\Delta fe^{-\frac{f}{\tau}}dV=\int_{M^n} \frac{\left|\nabla f\right|^2}{\tau}e^{-\frac{f}{\tau}}dV.
\end{equation}
By (\ref{1.2}) and substituting (\ref{1.11}) and (\ref{1.12}) into (\ref{1.10}),
\begin{align}
h'(\tau)&=(4\pi)^{-\frac{n}{2}}\tau^{-\frac{n}{2}-1}\int_{M^n}\frac{1-\tau}{\tau}\left(f-\left|\nabla f\right|^2\right)e^{-\frac{f}{\tau}}dV\nonumber\\
&=(4\pi)^{-\frac{n}{2}}\tau^{-\frac{n}{2}-2}\int_{M^n}(1-\tau)Re^{-\frac{f}{\tau}}dV.\label{1.13}
\end{align}
Since $R\geq0$, then we obtain the monotonicity of $h(\tau)$.

Next we turn to (\ref{1.9}). First we divide $h(\tau)$ into two parts:
$$h(\tau)=\int_{D(r)}(4\pi\tau)^{-\frac{n}{2}}e^{-\frac{f}{\tau}}dV+\int_{M^n\setminus D(r)}(4\pi\tau)^{-\frac{n}{2}}e^{-\frac{f}{\tau}}dV.$$
For the first part, we have
\begin{equation}\label{1.14}
\int_{D(r)}(4\pi\tau)^{-\frac{n}{2}}e^{-\frac{f}{\tau}}dV\leq(4\pi\tau)^{-\frac{n}{2}}V(r).
\end{equation}
Since $\boldsymbol{A}>0$, by Lemma \ref{l1.1}and \ref{l1.3}, we have
$$\frac{V(r)}{r^n}\longrightarrow \boldsymbol{A}\omega_n>0,\hspace{0.5em}\frac{rV'(r)}{V(r)}\longrightarrow n\hspace{0.5em}as\hspace{0.5em}r\longrightarrow+\infty,$$
thus for sufficiently large $r$ and $s\geq r$, we have 
\begin{align*}
\frac{sV'(s)}{V(s)}&\leq n+\epsilon\\
&=\left(n+\epsilon\right) \left(\frac{V(r)}{r^n}\frac{s^n}{V(s)}\right)^{-1}\left(\frac{V(r)}{r^n}\frac{s^n}{V(s)}\right)\\
&\leq(n+\epsilon)(1+\epsilon)\frac{V(r)}{r^n}\frac{s^n}{V(s)}\\
&\leq(n+\epsilon)\frac{V(r)}{r^n}\frac{s^n}{V(s)}
\end{align*}
if we adjust $\epsilon$ appropriately where $\epsilon$ is a small number and $\epsilon\longrightarrow0$ as $r\longrightarrow+\infty$. Hence we have
\begin{equation}\label{1.15}
V'(s)\leq(n+\epsilon)\frac{V(r)}{r^n}s^{n-1}.
\end{equation}
Now for the second part, by co-area formula, 
\begin{align}
\int_{M^n\setminus D(r)}(4\pi\tau)^{-\frac{n}{2}}e^{-\frac{f}{\tau}}dV&=\int_{r}^{+\infty}(4\pi\tau)^{-\frac{n}{2}}e^{-\frac{s^2}{4\tau}}V'(s)ds\nonumber\\
&\leq(n+\epsilon)\frac{V(r)}{r^n}(4\pi\tau)^{-\frac{n}{2}}\int_{r}^{+\infty}e^{-\frac{s^2}{4\tau}}s^{n-1}ds\nonumber\\
&\stackrel{t=\frac{s^2}{4\tau}}{=}\frac{n+\epsilon}{2}\pi^{-\frac{n}{2}}\frac{V(r)}{r^n}\int_{\frac{r^2}{4\tau}}^{+\infty}e^{-t}t^{\frac{n-2}{2}}dt.\label{1.16}
\end{align}
Note that
$$1=\int_{R^n}\pi^{-\frac{n}{2}}e^{-\left|x\right|^2}dx=\frac{n\omega_n}{2}\pi^{-\frac{n}{2}}\int_{0}^{+\infty}e^{-t}t^{\frac{n-2}{2}}dt,$$
then combining this equality and (\ref{1.16}) gives
\begin{equation}\label{1.17}
\int_{M^n\setminus D(r)}(4\pi\tau)^{-\frac{n}{2}}e^{-\frac{f}{\tau}}dV\leq\frac{n+\epsilon}{n\omega_n}\frac{V(r)}{r^n}.
\end{equation}
By (\ref{1.14}) and (\ref{1.17}), for all $\tau>0$ and sufficiently large $r$, we have
\begin{equation*}
h(\tau)\leq(4\pi)^{-\frac{n}{2}}\frac{V(r)}{\tau^{\frac{n}{2}}}+\frac{n+\epsilon}{n\omega_n}\frac{V(r)}{r^n}.
\end{equation*}
Hence setting $\tau\longrightarrow+\infty$ first and then $r\longrightarrow+\infty$ yields
$$\lim\limits_{\tau\longrightarrow+\infty}h(\tau)\leq\boldsymbol{A}.$$
For the inverse inequality, we only need to replace (\ref{1.15}) by
$$V'(s)\geq(n-\epsilon)\frac{V(r)}{r^n}s^{n-1}$$
and repeat the above procedures.
\end{proof}

\begin{remark}
In fact, for any non-compact $(M^n,g)$ with $Rc\geq0$, \cite[Lemma 8.10]{CCG} proved 
$$\lim\limits_{\tau\longrightarrow+\infty}\int_{M^n}(4\pi\tau)^{-\frac{n}{2}}e^{-\frac{d^2(\cdot, p)}{4\tau}}dV=\textbf{AVR}(M^n,g),$$ 
which is similar to (\ref{1.9}).
\end{remark}

In conclusion, we have the following result which can be viewed as a Ricci shrinker's version of the classical volume comparison theorem under $Rc\geq0$. 
\begin{theorem}[=Theorem \ref{t1.1*}]\label{t1.6}
For any non-compact Ricci shrinker $(M^n,g,f,p)$, 
\begin{equation}\label{1.19}
\boldsymbol{A}\leq e^{\boldsymbol{\mu}}(\leq1).
\end{equation}
The equality holds if and only if $(M^n,g,f,p)$ is isometric to the standard Gaussian soliton $\left( \mathbb{R}^n, dx^2, \frac{1}{4}\left|x\right|^2\right)$, otherwise $\boldsymbol{A}\leq1-\varepsilon$ for some small $\varepsilon=\varepsilon(n)>0$.
\end{theorem}
\begin{proof}
The case of asymptotic volume ratio $\boldsymbol{A}=0$ is trivial, therefore we only need to consider $\boldsymbol{A}>0$. By Proposition \ref{p1.4}, $h(\tau)$ is decreasing for $\tau\geq1$ and $\lim\limits_{\tau\longrightarrow+\infty}h(\tau)=\boldsymbol{A}$, hence 
$$e^{\boldsymbol{\mu}}=h(1)\geq \lim\limits_{\tau\longrightarrow+\infty}h(\tau)=\boldsymbol{A}.$$
If $e^{\boldsymbol{\mu}}=\boldsymbol{A}$, then $h'(\tau)=0$ for $\tau>1$, but in view of (\ref{1.13}), this forces $R\equiv0$, and hence the Ricci shrinker must be flat. 
\end{proof}
As a consequence of Theorem \ref{t1.6}, we have the following local gap  property which states that if the volume of some geodesic ball $B_r(p)$ is close to that of a standard Euclidean ball of the same radius, then the Ricci shrinker must be the standard Gaussian soliton.
\begin{corollary}[=Theorem \ref{t1.2*}]\label{c1.7}
For any Ricci shrinker $(M^n,g,f,p)$, there exists an $\varepsilon=\varepsilon(n)>0$ such that if
$$\omega_n^{-1}\varepsilon^n\left|B_{\frac{1}{\varepsilon}}(p)\right|\geq 1-\varepsilon,$$
then it must be flat.
\end{corollary}
\begin{proof}
If the above statement  were false, then there exists a sequence of non-flat Ricci shrinkers $(M_i,g_i,f_i,p_i)$ such that 
\begin{equation}\label{1.20}
\frac{\left| B_i(p_i)\right|_{g_i}}{\omega_ni^n}\geq 1-\frac{1}{i}.
\end{equation}
First of all, by the volume estimate of \cite[Lemma 2]{LW} and the assumption (\ref{1.20}), if $i\geq2\sqrt{n}$, then for some $c_1=c_1(n)>0$, we have 
\begin{equation*}
\boldsymbol{\mu}_i=\boldsymbol{\mu}(g_i)\geq -c_1.
\end{equation*}
Therefore, by the curvature estimate of Li-Wang\cite[Corollary 6.24]{LW2}, for some $c_2=c_2(c_1,n,\varepsilon)$ where $\varepsilon\in(0,1)$, 
\begin{equation*}
\int_{B_r(p)}R_i^{2-\varepsilon}dV_i\leq c_2r^{n-2+2\varepsilon},
\end{equation*}
then by \cite[Lemma 2]{LW} again, for some $c_3=c_3(c_2,n)$,
\begin{equation}\label{1.22}
\int_{B_r(p)}R_i\leq \left( \int_{B_r(p)}R_i^{2-\varepsilon}\right)^{\frac{1}{2-\varepsilon}}\left|B_r(p)\right|_{g_i}^{\frac{1-\varepsilon}{2-\varepsilon}}\leq c_3r^{\frac{n-2+2\varepsilon}{2-\varepsilon}}r^{\frac{(1-\varepsilon)n}{2-\varepsilon}}=c_3r^{n+1-\frac{4-3\varepsilon}{2-\varepsilon}}.
\end{equation}
By Lemma \ref{l1.3}, if $s\geq \sqrt{2(n+2)}$, then 
$$P_i'(s)=\frac{2X_i(s)}{s^{n+1}}\left\lbrace\frac{2(n+2)}{s^2}-1\right\rbrace\leq0.$$
According to Lemma \ref{l2.1}, $\left|B_s(p)\right|_{g_i}$ and $V_i(s)$ are almost the same, then by (\ref{1.22}), for some $c_4=c_4(c_3,n)$ and $\varepsilon=\frac{2}{3}$, we have
\begin{equation}\label{1.24}
\frac{X_i(s)}{s^{n+1}}\leq c_4s^{-\frac{3}{2}}.
\end{equation}
Now for $i\geq c_5$ where $c_5=c_5(n)$ is sufficiently large, by the proof of Lemma \ref{l1.3} and (\ref{1.24}), 
\begin{align}
\boldsymbol{A}_i\omega_n-P_i(i)&=\int_{i}^{+\infty}\frac{2X_i(s)}{s^{n+1}}\left\lbrace\frac{2(n+2)}{s^2}-1\right\rbrace ds\nonumber\\
&\geq-c_4\int_{i}^{+\infty}s^{-\frac{3}{2}}ds\nonumber\\
&=-\frac{2c_4}{\sqrt{i}}.\label{1.25}
\end{align}
Due to the definition of $P(r)$ and (\ref{1.25}), 
\begin{equation}\label{1.26}
	\boldsymbol{A}_i\geq \frac{V_i(i)}{\omega_ni^n}-\frac{4X_i(i)}{\omega_ni^{n+2}}-\frac{2c_4}{\sqrt{i}}.
\end{equation}
However, if $i=i(n)$ is sufficiently large, then (\ref{1.20}) and (\ref{1.26}) implies $A_i\geq 1-\varepsilon(n)$, but by Theorem \ref{t1.6}, this forces $(M_i,g_i,f_i,p_i)$ must be flat and contradicts our assumption at the beginning.
\end{proof}

\section{Rigidity and non-integrability of scalar curvature}
On any Ricci shrinker $(M^n,g,f,p)$, there holds the following Sobolev inequality which is fundamental for further analysis.
\begin{lemma}[\cite{LW}]\label{l2.2}
	For each compactly supported locally Lipschitz function $u$, 
	\begin{equation}\label{2.3}
		\left( \int_{M^n} u^{\frac{2n}{n-2}}\right)^{\frac{n-2}{n}}\leq Ce^{-\frac{2\boldsymbol{\mu}}{n}} \int_{M^n}\left( 4\left| \nabla u\right|^2+Ru^2\right) 
	\end{equation}
	for some dimensional constant $C = C(n)$.
\end{lemma}

As for the scalar curvature of a Ricci shrinker, most generally, we have the following fundamental estimate.
\begin{lemma}[\cite{CLY1}]\label{l2.3}
	For any non-compact and non-flat $(M^n,g,f,p)$, there exists a constant $C>0$ such that if $r\geq C$, then
	\begin{equation}\label{2.4}
		R(x)\geq \frac{1}{Cd^2 }\left( or\ equivalently\ \frac{1}{Cf}\right).
	\end{equation}
\end{lemma}

With the help of Lemma \ref{l2.2} and \ref{l2.3}, we have
\begin{proposition}\label{p3.1}
	For any non-compact and non-flat $(M^n,g,f,p)$, there exist constants $C_1,C_2>0$ which do not depend on $r$ such that if $r\geq C_1$, then
	\begin{equation}\label{3.1}
		\int_{B_{2r}}R\geq C_2\left| B_r\right|^{\frac{n-2}{n}}.
	\end{equation}
	Hence $\int_MR=+\infty$.
\end{proposition}
\begin{proof}
	By lemma \ref{l2.3}, there exists a constant $C_0$ such that if $r\geq C_0$, then we have
	\begin{equation}\label{3.2}
		R\geq \frac{1}{C_0r^2}.
	\end{equation}
	In view of \cite[Proposition 6]{LW}, there exists a constant $C'_{0}=C'_0\left( n,\boldsymbol{\mu}\right) $ such that if $r\geq C'_0$, then
	\begin{equation*}
		\left| B_r\right|\geq 2\left| B_{C_0}\right|.
	\end{equation*}
	Consequently, for $r\geq \max\left\lbrace C_0, C'_0\right\rbrace $, 
	\begin{align*}
		\int_{B_r}R\geq\int_{B_r\setminus B_{C_0}}R\geq \int_{B_r\setminus B_{C_0}}\frac{1}{C_0r^2}\geq  \frac{\left| B_r\right|}{2C_0r^2}.
	\end{align*}
	Hence
	\begin{align}\label{3.3}
		\int_{B_{2r}}R\geq \frac{\left| B_{2r}\right|}{2C_0(2r)^2}.
	\end{align}
	On the other hand, by Li-Wang's Sobolev inequality (\ref{2.3}), for some constant $C(n,\boldsymbol{\mu})$, there holds
	\begin{equation}\label{3.4}
		\left( \int_{M^n} u^{\frac{2n}{n-2}}\right)^{\frac{n-2}{n}}\leq C(n,\boldsymbol{\mu} ) \int_{M^n}\left(\left| \nabla u\right|^2+Ru^2\right) 
	\end{equation}
	for any compactly supported locally Lipschitz function $u$.
	
	Let $u(x)=u\left(r(x) \right)$ be a cut-off function such that
	\begin{equation*}
		u=
		\left\{
		\begin{array}{lr}
			1 , \quad x\in B_r,&\\
			0 , \quad x\in M\setminus B_{2r};&
		\end{array}
		\right.	 and \quad \left| \nabla u\right|^2\leq \frac{4}{r^2}.
	\end{equation*}
	Then by (\ref{3.4}), we have
	\begin{align*}
		\left| B_r\right|^{\frac{n-2}{n}}\leq C(n,\boldsymbol{\mu})\left( \frac{\left| B_{2r}\setminus B_r\right|}{(2r)^2}+\int_{B_{2r}}R\right)\leq C(n,\boldsymbol{\mu})\left( \frac{\left| B_{2r}\right|}{(2r)^2}+\int_{B_{2r}}R\right),
	\end{align*}
	therefore
	\begin{equation}\label{3.5}
		\frac{\int_{B_{2r}}R}{2C_0}\geq \frac{\left| B_r\right|^{\frac{n-2}{n}}}{2C_0C(n,\boldsymbol{\mu})}-\frac{\left| B_{2r}\right|}{2C_0(2r)^2}.
	\end{equation}
	Combining (\ref{3.3}) with (\ref{3.5}) yields
	\begin{equation}\label{3.6}
		\int_{B_{2r}}R\geq\frac{\frac{\left| B_r\right|^{\frac{n-2}{n}}}{2C_0C(n,\boldsymbol{\mu})}}{1+\frac{1}{2C_0}}.
	\end{equation}
	Hence we obtain (\ref{3.1}).
\end{proof}
As a consequence of the above result, we have

\begin{corollary}[\textbf{gap property}]\label{c3.2}
	For $(M^n,g,f,p)\in\mathcal{M}(\boldsymbol{\mu}_0)$, there exists a constant $\delta=\delta(n,\boldsymbol{\mu}_0)>0$ such that if
	$$\int_{M^n}R\leq \delta,$$
	then $(M^n,g,f,p)$ must be flat.
\end{corollary}
\begin{proof}
	By Proposition \ref{p3.1}, we only need to show $(M^n,g,f,p)$ must be non-compact if $\delta$ is small.  First of all, by Theorem 1.2 of \cite{FS}, the diameter $d(M)$ of any compact shrinkers $(M^n,g,f,p)$ has a universal lower bound, i.e. there exists a universal constant $D>0$ such that $d(M)\geq D$. Especially, we can choose $D=2$, i.e. $d(M)\geq2$.
	
	Suppose $(M^n,g,f,p)$ is compact, by (\ref{1.3}) and Lemma 2.4 of \cite{LW2},  we have
	$$\int_{M^n}R=\frac{n}{2}\left| M\right|\geq\frac{n}{2}\left| B_1(p)\right|\geq e^{-e^{n+8}}e^{\boldsymbol{\mu}}.$$
	Hence  choosing a $\delta<e^{-e^{n+8}}e^{\boldsymbol{\mu}_0}$ finishes the proof Corollary \ref{c3.2}.
\end{proof}

In the critical case $k=\frac{n}{2}$, generally, it's hard to estimate the growth rate of $\int_{D(r)}R^{\frac{n}{2}}$ with respect to $r$ like Proposition \ref{p3.1}. But we still have
\begin{proposition}\label{p3.3}
	For any non-compact and non-flat $(M^n,g,f,p)$, 
	$$\int_{M^n}R^{\frac{n}{2}}=+\infty.$$
\end{proposition}	

Proposition \ref{p3.3} follows from the following lemma directly.

\begin{lemma}\label{l3.4*}	
	If $\textbf{AVR}(M^n,g,f,p)>0$, there exists a constant $C>0$ such that
	\begin{equation}\label{3.7*}
		\int_{D(r)}R^{\frac{n}{2}}\geq C\ln r\quad as \quad r\longrightarrow+\infty.
	\end{equation}
	Moreover, $\textbf{AVR}(M^n,g,f,p)>0$ provided $\int_MR^k<+\infty$ for some $k\geq1$.
\end{lemma}
\begin{proof}
	First we assume $\textbf{AVR}(M^n,g,f,p)>0$.
	
	Since $\Delta f+R=\frac{n}{2}$, for all $s>0$, we have
	$$\frac{\fint_{D(s)}\Delta f}{s}+\frac{\fint_{D(s)}R}{r}=\frac{n}{2s}.$$
	By Green's formula, for some fixed $r_0>0$, we have
	$$\int_{r_0}^{r}\frac{\int_{\partial D(s)}\left|\nabla f\right| }{s\left|D(s)\right|}+\int_{r_0}^{r}\frac{\fint_{D(s)}R}{s}=\int_{r_0}^{r}\frac{n}{2s}=\frac{n}{2}\left(\ln r-\ln r_0\right).$$
	On $\partial D(s)$, we have $\left|\nabla f\right|\leq\sqrt{f}\leq \frac{s}{2}$, hence 
	$$\int_{r_0}^{r}\frac{\left| \partial D(s)\right|}{\left|D(s)\right|}\geq n\left(\ln r-\ln r_0\right)-2\int_{r_0}^{r}\frac{\fint_{D(s)}R}{s}.$$
	In view of (\ref{2.1}), $\left|B_s(p)\right|$ and $\left|D(s)\right|$ are almost the same if $s$ is large enough. So the assumption $\textbf{AVR}(M^n,g,f,p)>0$ implies there exists a universal constant $C_1>0$ such that for $r_0$ sufficiently large, there holds
	\begin{equation}\label{3.8*}
		\int_{r_0}^{r}\frac{\left| \partial D(s)\right|}{s^n}\geq nC_1\left(\ln r-\ln r_0\right)-2C_1\int_{r_0}^{r}\frac{\fint_{D(s)}R}{s}.
	\end{equation}
	However, by Lemma \ref{l1.1} and the invariance of single variable integral, we know 
	$$\int_{r_0}^{+\infty}\frac{\fint_{D(s)}R}{s}<+\infty.$$ 
	Therefore by (\ref{3.8*}), for $r_0$ sufficiently large, there exists a universal constant $C_2>0$ such that
	\begin{equation}\label{3.9*}
		\int_{r_0}^{r}\frac{\left| \partial D(s)\right|}{s^n}\geq C_2\left(\ln r-\ln r_0\right) .
	\end{equation}
	In light of (\ref{2.1}), for $s$ sufficiently large and $x\in\partial D(s)$, there holds
	$$s-\sqrt{2n}\leq d(x,p)\leq s+5n.$$
	By Lemma \ref{2.3}, there exists a constant $d_0$ such that if $d(x,p)\geq d_0$, then $R(x)\geq \frac{1}{d_0f}$.
	Let $r_0>d_0+\sqrt{2n}$, then $B_{d_0}(p)\subset D(r_0)$, thus in $D(s)\setminus D(r_0)$, we have $R\geq\frac{4}{d_0s^2}$.
	
	We recall that $\rho(x)=2\sqrt{f(x)}$, thus $\frac{1}{\left| \nabla \rho\right| }=\frac{\sqrt{f}}{\left| \nabla f\right|}=\frac{\sqrt{f}}{\sqrt{f-R}}\geq1$. In conclusion, by co-area formula, then for $r_2>r_1\geq r_0$, there holds
	\begin{align}\label{3.10*}
		\int_{D(r_2)}R^{\frac{n}{2}}\geq \int_{D(r_2)\setminus D(r_1)}R^{\frac{n}{2}}\geq\frac{2^n}{d_0^{\frac{n}{2}}}\int_{D(r_2)\setminus D(r_1)}\frac{1}{r_2^n}&=\frac{2^n}{d_0^{\frac{n}{2}}}\int_{r_1}^{r_2}\frac{\left|\partial D(s)\right|}{r_2^n}\frac{ds}{\left| \nabla \rho\right|}\nonumber\\
		&\geq\frac{2^n}{d_0^{\frac{n}{2}}}\int_{r_1}^{r_2}\frac{\left|\partial D(s)\right|}{r_2^n}ds.
	\end{align}
	Consequently by (\ref{3.10*}), for integers $i\geq 0$, we have
	\begin{align}\label{3.11*}
		\int_{D(2^{i+1}r_0)\setminus D(2^ir_0)}R^{\frac{n}{2}}\geq\frac{2^n}{d_0^{\frac{n}{2}}}\int_{2^ir_0}^{2^{i+1}r_0}\frac{\left|\partial D(s)\right|}{(2^{i+1}r_0)^n}ds\geq\frac{1}{d_0^{\frac{n}{2}}}\int_{2^ir_0}^{2^{i+1}r_0}\frac{\left|\partial D(s)\right|}{s^n}ds.
	\end{align}
	Finally, combining (\ref{3.9*}) and (\ref{3.11*}) yields the desired (\ref{3.7*}).
	
	The second conclusion just follows from the following claim.
	
	\textbf{Claim:} For non-compact $(M^n,g,f,p)$, if  $\int_{M}R^k<+\infty$ for some $k\geq1$, then $\textbf{AVR}(M^n,g,f,p)>0$.
	
	Since
	$$\int_{D(r)}R\leq\left(\int_{D(r)}R^k \right)^{\frac{1}{k}}\left| D(r)\right|^{\frac{k-1}{k}},$$
	then
	\begin{equation}\label{3.12*}
		\fint_{D(r)}R\leq\frac{\left(\int_{D(r)}R^k \right)^{\frac{1}{k}}}{\left| D(r)\right|^{\frac{1}{k}}}.
	\end{equation}
	According to Proposition 6 of \cite{LW}, for some $r_0=r_0(n)>0$ and $\epsilon_0=\epsilon_0(n)>0$, we have
	$$\left| D(r)\right|\geq\epsilon_0e^{\boldsymbol{\mu}} r\quad for \quad r\geq r_0.$$
	Then by (\ref{3.12*}), 
	\begin{equation}\label{3.13*}
		\fint_{D(r)}R\leq\frac{C\left( n,\boldsymbol{\mu},\left( \int_{M^n}R^k\right)^{\frac{1}{k}} \right) }{r^{\frac{1}{k}}}
	\end{equation}
	for sufficiently large $r>0$. Obviously, setting $c=\frac{r^2}{4}$ and using (\ref{3.13*}) yields $$\int_{n+2}^{+\infty}\frac{\overline{X}(c)}{c\overline{V}(c)}dc<+\infty,$$
	and hence by Lemma \ref{l1.1}, we conclude $\textbf{A}=\textbf{AVR}(M^n,g,f,p)>0$.
\end{proof}

\begin{corollary}[\textbf{gap property}]\label{c3.5}
	For $(M^n,g,f,p)\in\mathcal{M}(\boldsymbol{\mu}_0)$, there exists a constant $\delta=\delta(n,\boldsymbol{\mu}_0)>0$ such that if
	$$\int_{M}R^{\frac{n}{2}}\leq \delta,$$
	then $(M^n,g,f,p)$ must be flat.
\end{corollary}
\begin{proof}
	By Proposition \ref{p3.3}, $\int_{M}R^{\frac{n}{2}}<+\infty$ implies $(M^n,g,f,p)$ must be compact or flat. In the following, we will show that it can not be compact if $\delta(n,\boldsymbol{\mu}_0)$ is small enough, then the desired conclusion follows from this fact.
	
	If $(M^n,g,f,p)$ is compact, by (\ref{1.3}),
	$$\frac{n}{2}\left| M^n\right|=\int_{M^n}R\leq\left(\int_{M^n}R^{\frac{n}{2}} \right)^{\frac{2}{n}}\left| M^n\right|^{\frac{n-2}{n}},$$
	then by the proof of Corollary \ref{c3.2}, 
	\begin{equation*}
		\int_{M^n}R^{\frac{n}{2}}\geq \left( \frac{n}{2}\right)^{\frac{n}{2}}\left| M^n\right|\geq\left( \frac{n}{2}\right)^{\frac{n}{2}}\left| B_1(p)\right|\geq\left( \frac{n}{2}\right)^{\frac{n}{2}} e^{-e^{n+8}}e^{\boldsymbol{\mu}}.
	\end{equation*}
	Consequently, choosing a $\delta<\left( \frac{n}{2}\right)^{\frac{n}{2}} e^{-e^{n+8}}e^{\boldsymbol{\mu}_0}$ finishes the proof of Corollary \ref{c3.5}.
\end{proof}

\begin{proposition}\label{p3.6}
	For any non-compact and non-flat $(M^n,g,f,p)$ with $k>\frac{n}{2}$, there exists a constant $C>0$ such that
	\begin{equation*}
		\int_{M^n}R^k\geq C.
	\end{equation*}
\end{proposition}
\begin{proof}	
	By Lemma \ref{l2.3}, there exists a constant $\rho>2$ such that on $B_{2\rho}\setminus B_\rho$, 
	$$R\geq\frac{1}{\rho(2\rho)^2}=\frac{1}{4\rho^3},$$
	then
	$$\left( \int_{B_{2\rho}\setminus B_\rho}R^k\right)^{\frac{1}{k}} \geq\frac{\left|B_{2\rho}\setminus B_\rho \right|^{\frac{1}{k}}}{4\rho^3}.$$
	On the other hand, by Lemma 2.4 of \cite{LW2}, for $B_1(q)\subset B_{2\rho}\setminus B_\rho\subset B_{2\rho}$, 
	$$\left| B_1(q)\right|\geq e^{-e^{n+8\rho}}e^{\boldsymbol{\mu}},$$
	and hence
	$$\int_{M^n}R^k\geq\int_{B_{2\rho}\setminus B_\rho}R^k\geq\frac{e^{-e^{n+8\rho}}e^{\boldsymbol{\mu}}}{\left(4\rho^3 \right)^k }\stackrel{\vartriangle}{=}C.$$
\end{proof}

\begin{corollary}[\textbf{gap property}]\label{c3.7}
	For $(M^n,g,f,p)\in\mathcal{M}(\boldsymbol{\mu}_0)$ and $k>\frac{n}{2}$, there exists a constant $\delta=\delta(n,\boldsymbol{\mu}_0,k)>0$ such that if
	\begin{equation}\label{3.22}
		\int_{M^n}R^k\leq \delta,
	\end{equation}
	then $(M^n,g,f,p)$ must be flat.
\end{corollary}
\begin{proof}
	By Proposition \ref{p3.6},  we only need to show $(M^n,g,f,p)$ must be non-compact and the remain steps are the same as in Corollary \ref{c3.2}.
\end{proof}

With the help of the above arguments, now we are ready to conclude the main theorems.

\begin{theorem}[\textbf{=Theorem \ref{t1.3*}, Non-integrability}]\label{t3.8}
	For any non-compact and non-flat $(M^n,g,f,p)$, we have
	\begin{equation}\label{3.16*}
		\int_{M^n}R^k=+\infty\quad if \quad 1\leq k\leq \frac{n}{2},
	\end{equation}
	and there exists a constant $C=C(n,\boldsymbol{\mu},k)>0$ such that
	\begin{equation}\label{3.17*}
		\int_{M^n}R^k\geq C\quad if \quad k>\frac{n}{2}.
	\end{equation}
	Furthermore, the critical index $k=\frac{n}{2}$ is sharp.
\end{theorem}
\begin{proof}
	In view of Proposition \ref{p3.1}, \ref{p3.3} and \ref{p3.6}, to prove (\ref{3.16*}) and (\ref{3.17*}), we only need to consider the case of $1<k<\frac{n}{2}$ and our conclusions follow from a contradiction. If
	$$\int_{M^n}R^k<+\infty,$$
	then by Lemma \ref{l3.4*}, we know $\textbf{AVR}(M^n,g,f,p)>0$. Furthermore, by Proposition \ref{p3.1}, for sufficiently large $r>0$, there holds
	$$C\left|B_r\right|^{\frac{n-2}{n}}\leq\int_{B_{2r}}R\leq\left(\int_{B_{2r}}R^k\right)^{\frac{1}{k}}\left|B_{2r}\right|^{\frac{k-1}{k}}$$ for some constant $C>0$.
	Thus we have
	$$\left(\int_{B_{2r}}R^k\right)^{\frac{1}{k}}\geq\frac{C\left|B_r\right|^{\frac{n-2}{n}}}{\left|B_{2r}\right|^{\frac{k-1}{k}}}.$$
	Since $\textbf{A}=\textbf{AVR}(M^n,g,f,p)>0$, then $\left|B_r\right|\longrightarrow A\omega_nr^n$ as $r\longrightarrow+\infty$. Finally by the fact $\frac{n-2}{n}>\frac{k-1}{k}$, we conclude (\ref{3.16*}) for $1<k<\frac{n}{2}$.
	
	To see the sharpness of $k=\frac{n}{2}$, we need the following lemma.
	\begin{lemma}\label{l3.10}
		There exists a constant $C=C(n)$ such that for $r>C$, we have
		\begin{equation}\label{3.19}
			\left|D(r)\right|'\leq \frac{C\left|D(r)\right|}{r}.
		\end{equation}
		Especially, if $\textbf{AVR}(M^n,g,f,p)>0$, then for sufficiently large $r>0$, we have
		\begin{equation}\label{3.20}
			\left|D(r)\right|'\leq Cr^{n-1}.
		\end{equation}
	\end{lemma}
	\begin{proof}
		By (6.24) and (6.25) of \cite{MW1}, we have
		$$\left|D(r+1)\right|-\left|D(r)\right|\leq C_1\frac{\left|D(r)\right|}{r}$$
		whenever $r\geq C_1$ for some dimensional constant $C_1=C_1(n)$. Then by integral mean value theorem, there exists a $r_0\in\left[r,r+1\right]$ such that
		$$\left|D(r_0)\right|'\leq C_1\frac{\left|D(r)\right|}{r}.$$
		Moreover, it's easy to see that for each $r\geq C_1$, we can choose a $r_0$ depending continuously on $r$, hence for every $r_0>C_1+1$, there exists a $r=r(r_0)$ such that $r_0\in\left[r,r+1\right]$ and
		$$\left|D(r_0)\right|'\leq C_1\frac{\left|D(r)\right|}{r}\leq C_1\frac{\left|D(r_0)\right|}{r_0}\frac{r_0}{r}.$$
		Consequently, for all sufficiently large $r>C_1+1$, we have
		$$\left|D(r)\right|'\leq\frac{2C_1\left|D(r)\right|}{r}.$$
	\end{proof}
	With the help of Lemma \ref{l3.10}, now we can show for $k>\frac{n}{2}$, there exists a shrinker on which $R^k$ is integrable.
	By \cite{FIK}, there exist a Ricci shrinker which has Euclidean volume growth and quadratic curvature decay, then by Lemma \ref{l3.10} and co-area formula, we know $R^k$ is integrable.
\end{proof}

\begin{theorem}[\textbf{=Theorem \ref{t1.4*}, Gap Theorem}]\label{t3.10}
	For any compact or non-compact $(M^n,g,f,p)\in\mathcal{M}(\boldsymbol{\mu}_0)$ and $k\geq 1$, there exists a constant $\delta=\delta(n,\boldsymbol{\mu}_0,k)>0$ such that if
	\begin{equation}\label{3.18*}
		\int_{M^n}R^k\leq \delta,
	\end{equation}
	then it must be flat.
\end{theorem}
\begin{proof}
	By Theorem \ref{t3.8}, choosing a sufficiently small $\delta=\delta(n,\boldsymbol{\mu},k)>0$ in (\ref{3.18*}) implies $(M^n,g,f,p)$ must be compact or flat. Then our conclusion follows from the same arguments as in Corollary \ref{c3.5}.
\end{proof}

\section{$\varepsilon$-regularity theorems of local entropy and curvature}\label{sec4}
Due to B. Wang \cite{Wb,WB2}, the local functionals of Perelman are of vital importance to the study of local geometry of Riemannian manifolds. First of all, let us introduce the relevant definitions.  Let $\Omega$ be a connected, open subset of $(M^n,g)$ with smooth boundary. For $\varphi\in W^{1,2}_0(\Omega),\varphi\geq0$ and $\int_\Omega\varphi^{2}=1$, as usual, we define the local entropy functionals
	\begin{align*}
		&\boldsymbol{\mathcal{W}}(\Omega,g,\varphi,\tau):=-n-n\log\sqrt{4\pi\tau}+\int_\Omega\tau\left(R \varphi^2+4\left| \nabla\varphi\right|^2 \right)-2\varphi^2\log\varphi dV, \\
		&\boldsymbol{\mu}(\Omega,g,\tau):=\inf\limits_{\varphi}\boldsymbol{\mathcal{W}}(\Omega,g,\varphi,\tau), \\
		&\boldsymbol{\nu}(\Omega,g,\tau):=\inf\limits_{s\in\left(0,\tau\right] }\boldsymbol{\mu}(\Omega,g,s).
	\end{align*}
Especially, the local $\boldsymbol{\nu}$ entropy is closely related to the isoperimetric constant and volume ratio. For more information, we refer to \cite{Wb,WB2}. On Ricci shrinkers, the smallness of $\boldsymbol{\nu}$ implies a pseudo-locality theorem(cf. \cite[Theorem 24 and 25]{LW}). Besides, since Ricci shrinkers are natural generalizations of Einstein manifolds, hence one can expect that there are some good $\varepsilon$ regularities of curvature on Ricci shrinkers.
The following $\varepsilon$-regularity theorem is an application of the two-sided pseudo-locality theorem of Ricci shrinkers(cf. Li-Wang \cite[Theorem 1.6]{LW2}). 
\begin{theorem}[\textbf{=Theorem \ref{t1.5*}}]\label{t4.1}
For any $x\in (M^n,g,f,p)$ and $r>0$, there exists a constant $\varepsilon=\varepsilon(n)>0$ such that if 
\begin{equation*}
\boldsymbol{\nu}\left(B\left(x,r\right) ,r^2\right)\geq -\varepsilon,
\end{equation*}
 then
\begin{equation*}
\left|Rm\right|(x)\leq \varepsilon^{-1}r^{-2}.
\end{equation*}
\end{theorem}
\begin{remark}
Under similar conditions, \cite[Theorem 25]{LW} proved $\left|Rm\right|(x)\leq C(n)D$ where $D=d(x,p)+\sqrt{2n}$. 
\end{remark}

\begin{proof}
Our proof follows from a standard blowup argument. By scaling invariance of $\boldsymbol{\nu}\left(B\left(x,r\right) ,r^2\right)$, up to rescaling $g$ by $r^{-2}$, we may assume $r=1$. Suppose the statement were false, then there exists a sequence of non-flat Ricci shrinkers $\left( M^n_k,g_k,f_k,p_k\right)$ and points $x_k\in \left( M^n_k,g_k,f_k,p_k\right) $, where $k\in 1,2,3,\cdots$, such that
\begin{equation*}
\boldsymbol{\nu}_{g_k}\left(B_{g_k}(x_k,1),1\right)\geq -\frac{1}{k} \quad but\quad \left|Rm\right|_{g_k}(x_k)\geq k.
\end{equation*}
Let $P_k^2=\left|Rm\right|_{g_k}(x_k)\geq k$, and scale $g_k$ by $\hat{g}_k=P_k^2g_k$, then by the scaling invariance of $\boldsymbol{\nu}$ and $\left|Rm\right|d^2$,
\begin{equation}\label{4.3}
\boldsymbol{\nu}_{\hat{g}_k}\left(B_{\hat{g}_k}\left(x_k,P_k\right),P_k^2\right)\geq-\frac{1}{k}\quad and \quad \left|Rm\right|_{\hat{g}_k}(x_k)=1.
\end{equation}
Let $L_k>0$ be a sequence of numbers such that $P_k= 4L_k\longrightarrow+\infty$. Let $y_k\in B_{\hat{g}_k}(x_k,4L_k)=\Omega_k$ such that  $\left|Rm\right|_{\hat{g}_k}(\cdot)d^2_{\hat{g}_k}(\cdot,\partial\Omega_k)$ achieves its maximum $F_k^2$ in $\Omega_k$ at $y_k$. Hence
\begin{equation*}
F_k^2=\left|Rm\right|_{\hat{g}_k}(y_k)d^2_{\hat{g}_k}(y_k,\partial\Omega_k)\geq \left|Rm\right|_{\hat{g}_k}(x_k)d^2_{\hat{g}_k}(x_k,\partial\Omega_k)=16L_k^2.
\end{equation*}
Let $Q_k^2=\left|Rm\right|_{\hat{g}_k}(y_k)$, then after scaling $\hat{g}_k$ by $\bar{g}_k=Q_k^2\hat{g}_k$, 
\begin{equation}\label{4.4}
\boldsymbol{\nu}_{\bar{g}_k}\left(B_{\bar{g}_k}\left(x_k,P_kQ_k\right),P_k^2Q_k^2\right)\geq-\frac{1}{k},\quad \left|Rm\right|_{\bar{g}_k}(y_k)=1\quad and \quad d_{\bar{g}_k}\left( y_k,\partial\Omega\right)=F_k\geq 4L_k.
\end{equation}
Now for any $x\in B_{\bar{g}_k}\left(y_k,2L_k\right)$, it's easy to know
$$d_{\bar{g}_k}\left( x,\partial\Omega\right)\geq F_k-2L_k\geq 2L_k,$$
and consequently,
\begin{equation}\label{4.5}
 \left|Rm\right|_{\bar{g}_k}(x)\leq \frac{\left|Rm\right|_{\bar{g}_k}(y_k)d^2_{\bar{g}_k}(y_k,\partial\Omega_k)}{d^2_{\bar{g}_k}\left( x,\partial\Omega\right)}\leq\frac{F_k^2}{\left(F_k-2L_k\right)^2}\leq 4.
\end{equation}
Notice that 
$$Q_k^2=\left|Rm\right|_{\hat{g}_k}(y_k)\geq \frac{16L_k^2}{d^2_{\hat{g}_k}(y_k,\partial\Omega_k)}\geq 1\quad and \quad B_{\bar{g}_k}\left(y_k,2L_k\right)\subset B_{\bar{g}_k}\left(x_k,4L_kQ_k\right)= B_{\bar{g}_k}\left(x_k,P_kQ_k\right),$$
therefore by (\ref{4.4}) and the monotonicity of $\boldsymbol{\nu}$ in \cite[Proposition 2.1, 3.2]{Wb}, for all $\bar{x}\in B_{\bar{g}_k}\left(y_k,L_k\right)$,
\begin{equation}\label{4.6}
\boldsymbol{\nu}_{\bar{g}_k}\left( B_{\bar{g}_k}\left(\bar{x},L_k\right), L_k^2\right)\geq\boldsymbol{\nu}_{\bar{g}_k}\left(B_{\bar{g}_k} \left(y_k,2L_k\right), 4L^2_k\right)\geq-\frac{1}{k}.
\end{equation}
By (\ref{4.5}), for all $\bar{y}\in B_{\bar{g}_k}\left(\bar{x},L_k\right)$, we have $ \left|Rm\right|_{\bar{g}_k}(\bar{y})\leq4$, then in light of the estimate of lower bound of volume ratio in \cite[Theorem 3.3]{Wb}, for some constant $\theta=\theta(n)>0$ and all $0<r\leq 1$,
\begin{equation}\label{4.7}
\left|B_{\bar{g}_k}\left(\bar{x},r\right)\right|_{\bar{g}_k}\geq\theta\omega_nr^n.
\end{equation}
Then by (\ref{4.5}), (\ref{4.6}), (\ref{4.7}) and the estimate of lower bound of injective radius by Cheeger-Gromov-Taylor\cite{CGT}, there exists a uniform constant $c_0=c_0(n)$ such that
\begin{equation}\label{4.8}
inj_{\bar{g}_k}(y_k)\geq c_0.
\end{equation}

Now by reverting $\bar{g}_k$ to $g_k$ and setting $r_k=\frac{1}{P_kQ_k}$, we have
\begin{equation}\label{4.9}
\frac{\left|B_{g_k}\left(\bar{x},r_k\right)\right|_{g_k}}{\omega_nr_k^n}\geq\theta\quad and \quad \left|Rm\right|_{g_k}\leq4P_k^2Q_k^2=\frac{4}{r_k^2}\quad on \quad B_{g_k}(\bar{x},r_k).
\end{equation}
Let $\alpha_k=\min\left\lbrace \frac{1}{2},\theta\omega_n\right\rbrace $ and by the two-sided pseudo-locality theorem of Ricci shrinkers(cf. Li-Wang \cite[Theorem 1.6]{LW2}) and (\ref{4.9}), for some $\epsilon_k=\epsilon_k(n,\alpha_k)=\epsilon_k(n)$,
\begin{equation}\label{4.10}
\left|Rm\right|_{g_k}(y,t)\leq \left(\epsilon_kr_k\right)^{-2}, (y,t)\in B_{g_k}(\bar{x},(1-\alpha_k)r_k)\times\left( \left[-\left(\epsilon_kr_k\right)^2,\left(\epsilon_kr_k\right)^2\right]\bigcap(-\infty,1)\right) .
\end{equation}
Notice that $\epsilon_kr_k=\frac{\epsilon_k}{P_kQ_k}\longrightarrow0$, hence we can consider the following well-defined parabolic rescaled Ricci flows
\begin{equation}\label{4.11}
\tilde{g}_k(t)=\frac{P_k^2Q_k^2}{2\epsilon_k^2}g_k\left(\frac{2\epsilon_k^2t}{P_k^2Q_k^2}\right), g_k(0)=g_k, t\in\left[-\frac{1}{2},\frac{1}{2}\right],
\end{equation}
where $g_k(t)$ is the Ricci flow induced by the Ricci shrinker $\left( M^n_k,g_k,f_k,p_k\right)$.  By (\ref{4.10}), under $\tilde{g}_k(t)$, we have
\begin{equation}\label{4.12}
\left|Rm\right|_{\tilde{g}_k(t)}(\bar{x})\leq\frac{2\epsilon_k^2\left(\epsilon_kr_k\right)^{-2}}{P_k^2Q_k^2}=2.
\end{equation}
Recall that by our choice, $\bar{x}$ is an arbitrary point of $B_{\bar{g}_k(0)}\left(y_k,L_k \right)=B_{\tilde{g}_k(0)}\left(y_k,\frac{L_k}{\sqrt{2}\epsilon_k} \right)$. Note that $\tilde{g}_k(0)=\frac{\bar{g}_k}{2\epsilon_k^2}$, hence we can argue as to deduce (\ref{4.8}) to get $inj_{\tilde{g}_k}(y_k)\geq c_1(n)>0$.  Consequently, by the local version of Hamilton's compactness theorem as well as Shi's local derivative estimates \cite{Shi}, there exists a sub-sequence of $\tilde{g}_k(t)$ such that
\begin{equation}\label{4.14}
	\left(B_{\tilde{g}_k(0)}\left(y_k,\frac{L_k}{2\epsilon_k}\right),\tilde{g}_k(t),y_k\right)\xrightarrow{C^\infty-Cheeger-Gromov}\left(M^n_\infty,g_\infty(t),y_\infty\right).
\end{equation}
Especially, by (\ref{4.5}), $\left(M^n_\infty,g_\infty(0),y_\infty\right)$ has bounded curvature. Let $\overset{\infty}{\bigcup\limits_{i=1}}\Omega_i$ be an exhaustion of $M^n_\infty$ with smooth boundaries. Fixing $i$ and $\tau>0$, by (\ref{4.6}) and the continuity of $\boldsymbol{\mu}$ in \cite[Proposition 2.3 and Corollary 2.5]{Wb}, 
$$\boldsymbol{\mu}\left(\Omega_i,g_\infty,\tau\right)=\lim\limits_{j\longrightarrow+\infty}\boldsymbol{\mu}\left(\Omega_i,g_j,\tau\right)\geq\lim\limits_{j\longrightarrow+\infty}\boldsymbol{\nu}\left(\Omega_i,g_j,\tau\right)\geq0,$$
and hence 
$$\boldsymbol{\mu}_{g_\infty}\left(M^n_\infty,g_\infty(0),\tau\right)\geq0.$$
Thus , it's easy to see,
\begin{equation}
\boldsymbol{\nu}_{g_\infty}\left(M^n_\infty,g_\infty(0),\tau\right)=\inf\limits_{s\in\left( 0,\tau\right] }\boldsymbol{\mu}_{g_\infty}\left(M^n_\infty,g_\infty(0),\tau\right)\geq0.
\end{equation}
However, by the rigidity theorem of \cite[Proposition 3.2]{WB2}, $\left(M^n_\infty,g_\infty(0)\right)$ must be isometric to the standard Euclidean space which contradicts the fact of (\ref{4.4}) that $\left|Rm\right|_{\tilde{g}_k(0)}(y_k)=2\epsilon_k^2\left|Rm\right|_{\bar{g}_k}(y_k)=2\epsilon_k^2(n)>0$. Hence we complete the proof.
\end{proof}

By the same arguments as above, there holds another form of Theorem \ref{t4.1}.
\begin{corollary}
For any $x\in (M^n,g,f,p)$, $\delta>0$ and $r>0$, there exists a constant $\varepsilon=\varepsilon(n,\delta)>0$ such that if 
\begin{equation*}
	\boldsymbol{\nu}\left(B\left(x, \varepsilon^{-1}r\right), \varepsilon^{-2}r^2\right)\geq -\varepsilon,
\end{equation*}
then
\begin{equation*}
	\left|Rm\right|(x)\leq \delta r^{-2}.
\end{equation*}
\end{corollary}

Especially, if $x=p$, we have the following gap property.
\begin{proposition}\label{p4.3}
For any Ricci shrinker $(M^n,g,f,p)$, there exists a constant $\varepsilon=\varepsilon(n)$ such that if
\begin{equation*}
\boldsymbol{\nu}\left(B\left(p,\varepsilon^{-1}\right),\varepsilon^{-2}\right)\geq -\varepsilon,
\end{equation*}
then it must be flat.
\end{proposition}
\begin{proof}
This is a direct consequence of the pseudo-locality theorem of Ricci shrinkers(cf.\cite[section 10]{LW}). By \cite[Corollary 8]{LW}, for some small $\varepsilon=\varepsilon(n)>0$, there exists a constant $C=C(n)$ such that for all $t\in\left[ 0,1\right)$, 
$$\left|Rm\right|(\psi^t(p))\leq C(1-t).$$ 
However, by the self-similarity of Ricci shrinkers, $\psi^t(p)=p$ for all $t\in\left[ 0,1\right)$. Then setting $t\longrightarrow1$ implies $\left|Rm\right|(p)=0$, hence $(M^n,g,f,p)$ must be flat.
\end{proof}

As a consequence of Theorem \ref{t4.1}, we conclude a Harnark inequality for scalar curvature.
\begin{theorem}\label{t4.4}
For any $x\in(M^n,g,f,p)$ and $r>0$, there exists a constant $\varepsilon=\varepsilon(n)>0$ such that if
\begin{equation*}
	\boldsymbol{\nu}\left(B\left(x,3r\right),9r^2\right)\geq -\varepsilon,
\end{equation*}
then for any $y\in B\left(x,r\right)$,
\begin{equation*}
\frac{R(y)}{R(x)}\leq\frac{1}{\varepsilon}.
\end{equation*}
\end{theorem}
\begin{proof}
Up to rescaling the metric by $r^{-2}$, we may assume $r=1$. Suppose the statement were false, then there exists a sequence of non-flat Ricci shrinkers $\left( M^n_k,g_k,f_k,p_k\right)$ and points $x_k\in \left( M^n_k,g_k,f_k,p_k\right) $, where $k\in 1,2,3,\cdots$, such that for some $y_k\in B_{g_k}(x_k,k)$,
\begin{equation}\label{4.17}
	\boldsymbol{\nu}_{g_k}\left(B_{g_k}(x_k,3),9\right)\geq -\frac{1}{k}, \quad but\quad \frac{R_{g_k}(y_k)}{R_{g_k}(x_k)}\geq k.
\end{equation}
In light of Theorem \ref{t4.1}, for $k$ large enough, we have $R_{g_k}(y_k)\leq C(n)<+\infty$ and $\left| Rm\right|_{g_k}\leq C(n)$ on $B_{g_k}(x_k,2)$, and then by (\ref{4.17}), we know 
\begin{equation}\label{4.18}
r_k=R_{g_k}(x_k)\longrightarrow 0\quad as \quad k\longrightarrow+\infty.
\end{equation}
Rescaling the metrics $g_k$ by $\bar{g}_k=r_k^{-1}g_k$, we have
\begin{equation}\label{4.19}
		\boldsymbol{\nu}_{\bar{g}_k}\left( B_{\bar{g}_k}\left(y_k, r_k^{-\frac{1}{2}}\right) ,r_k^{-1}\right)\geq\boldsymbol{\nu}_{\bar{g}_k}\left(B_{\bar{g}_k}\left( x_k,2r_k^{-\frac{1}{2}}\right) ,4r_k^{-1}\right)\geq -\frac{1}{k},\quad 
\end{equation}
\begin{equation}\label{4.20}
 \left| Rm\right|_{g_k}\leq C(n)r_k^{-1}\quad on\quad B_{\bar{g}_k}\left(y_k, r_k^{-\frac{1}{2}}\right),\quad R_{\bar{g}_k}(y_k)\geq k.
\end{equation}
Then by the same arguments as in the proof of Theorem \ref{t4.1}, we can choose a sub-sequence of parabolic rescaled Ricci flows $\tilde{g}_k(t)$ such that
\begin{equation*}
	\left(B_{\tilde{g}_k(0)}\left(y_k,r_k^{-\frac{1}{2}}\right),\tilde{g}_k(t),y_k\right)\xrightarrow{C^\infty-Cheeger-Gromov}\left(M^n_\infty,g_\infty(t),y_\infty\right),
\end{equation*}
where the limit space is isometric to the standard Euclidean space which contradicts the fact that $ R_{\bar{g}_k}(y_k)\geq k$ of (\ref{4.20}). Hence the proof is complete.
\end{proof}
\begin{remark}
In fact, by the proof, if we replace $R(y)/R(x)$ by $\left|Rm\right|(y)/R(x)$, Theorem \ref{t4.4} is still valid.
\end{remark}

The following $\varepsilon$-regularity theorem is the key point of the next section. Notice that it is independent of the location of $x$. 
\begin{theorem}[\textbf{=Theorem \ref{t1.6*}}]\label{t5.4}
	For any $x\in(M^4,g,f,p)$ and $r\in\left(0,1\right]$, there exists a constant $\varepsilon=\varepsilon(\boldsymbol{\mu})>0$ such that if
	\begin{equation}\label{5.8}
		\int_{B\left(x,2r\right)}R^{2}dV\leq\varepsilon,\quad f(x)\geq \varepsilon^{-1}r^{-2},
	\end{equation}
	then
	\begin{equation}\label{5.9}
		\left|Rm\right|(x)\leq \varepsilon^{-1}r^{-2}.
	\end{equation}
\end{theorem}
\begin{remark}
	Indeed, the above theorem localizes the main theories of Munteanu-Wang \cite{MW1*}, where the authors bounded the Riemannian curvature by global bounded scalar curvature. Specifically, given a point $x\in M^4$, we assume that $R(y)\leq K$ for any $y\in B\left( x,2r\right)$ and $f(x)\geq \varepsilon^{-1}r^{-2}$. By the non-expanding estimate of volume ratio in \cite[Theorem 1.2]{LW2}, we have $\left|B\left(x,2r\right)\right|\leq Cr^4$. Consequently, choosing a sufficiently small $r=r(\varepsilon, K)\leq 1$ ensures the validity of (\ref{5.8}), hence $	\left|Rm\right|(x)\leq C(\varepsilon,K)$.
\end{remark}

To prove this theorem, first we need to estimate the lower bound of volume ratio.
\begin{lemma}\label{l5.5}
	For any $x\in(M^n,g,f,p)$ and $r>0$, there exists two constants $\varepsilon=\varepsilon(n,\boldsymbol{\mu})>0$ and $C=C(n,\boldsymbol{\mu})>0$ such that if 
	\begin{equation*}
		\int_{B\left(x,r\right)}R^{\frac{n}{2}}dV\leq\varepsilon,
	\end{equation*}
	then
	\begin{equation*}
		\frac{\left|B(x,r)\right|}{\omega_nr^n}\geq C(n,\boldsymbol{\mu}).
	\end{equation*}
\end{lemma}
\begin{proof}
	The proof is similar to that of \cite[Theorem 23]{LW}. Let $r_0\in[0,r]$ such that $\inf\limits_{s\in[0,r]}s^{-n}\left|B(x,r)\right|$ is achieved at $r_0$. 
	
	If $r_0=0$, then by case 1 therein, we have
	\begin{equation*}
		\frac{\left|B(x,r)\right|}{\omega_nr^n}\geq1.
	\end{equation*}
	
	If $r_0>0$, we choose a non-increasing smooth cut-off function $\phi$ on real line such that $\phi=1$ on $\left(-\infty, \frac{1}{2}\right]$ and $\phi=0$ on $\left[0,+\infty\right).$ Let $u(\cdot)=\phi\left(\frac{d(\cdot,x)}{r_0}\right)$. Now by the Sobolev inequality (\ref{2.3}),
	\begin{align}
		\left|B\left( x,\frac{r_0}{2}\right) \right|^{\frac{n-2}{n}}&\leq C(n)e^{-\frac{2\boldsymbol{\mu}}{n}}\int\left(4\left|\nabla u\right|^2+Ru^2\right)\nonumber\\
		&\leq C(n)e^{-\frac{2\boldsymbol{\mu}}{n}}\left(r_0^{-2}\left|B\left( x,r_0\right) \right|+\int_{B(x,r_0)} Ru^2\right)\nonumber\\
		&\leq C(n)e^{-\frac{2\boldsymbol{\mu}}{n}}\left(r_0^{-2}\left|B\left( x,r_0\right) \right|+\left( \int_{B(x,r)} R^{\frac{n}{2}}\right)^{\frac{2}{n}}\left|B\left( x,r_0\right) \right|^{\frac{n-2}{n}}\right)\nonumber\\
		&\leq C(n)e^{-\frac{2\boldsymbol{\mu}}{n}}\left(r_0^{-2}\left|B\left( x,r_0\right) \right|+\varepsilon^{\frac{2}{n}}\left|B\left( x,r_0\right) \right|^{\frac{n-2}{n}}\right).\label{5.13}
	\end{align}
	By our choice of $r_0$, 
	\begin{equation}\label{5.14}
		\left|B\left( x,\frac{r_0}{2}\right) \right|\geq 2^{-n}\left|B\left( x,r_0\right) \right|,\quad r^{-n}\left|B\left(x,r\right)\right|\geq r_0^{-n}\left|B\left(x,r_0\right)\right|.
	\end{equation}
	As a consequence, combining (\ref{5.13}), (\ref{5.14}) and choosing an $\varepsilon(n,\boldsymbol{\mu})$ small enough completes the proof.
\end{proof}

Besides the above lemma, we also need to control the growth of $\left|Rm\right|$. 
\begin{lemma}\label{l5.8}
	For any $x\in(M^4,g,f,p)$, $r\in\left(0,1\right]$ and $\delta>0$, there exists a constant $\varepsilon=\varepsilon(\boldsymbol{\mu},\delta)>0$ such that if
	\begin{equation*}
		\int_{B\left(x,2r\right)}R^{2}dV\leq\varepsilon, \quad f(x)\geq \varepsilon^{-1} r^{-2},
	\end{equation*}
	then
	$$\sup\limits_{y\in B(x,r)}\frac{\left|Rm\right|}{f}(y)\leq \delta.$$
\end{lemma}
\begin{proof}
	Up to rescaling the metric by $r^{-2}$, we may assume $r=1$. Suppose the statement were false, then there exists a sequence of non-flat Ricci shrinkers $\left( M^4_k,g_k,f_k,p_k\right)$, points $x_k\in \left( M^4_k,g_k,f_k,p_k\right)$ and $y_k\in B(x_k,1)$, where $k\in 1,2,3,\cdots$, such that
	\begin{equation*}
		\int_{B_{g_k}(x_k,2)}R_{g_k}^2\leq \frac{1}{k},\quad f_k(y_k)\geq k \quad but\quad \left|Rm\right|_{g_k}(y_k)\geq \delta f_k(y_k).
	\end{equation*}	
	Here we use the fact that $f_k(x_k)$ and $f_k(y_k)$ are almost the same at infinity. If we set $s_k=f_k(y_k), \tilde{g}_k=s_kg_k, \tilde{f}_k=f_k-f_k(y_k)$, then by the weak-compactness of \cite[Theorem 3.4]{LW3}, we can find a convergent sub-sequence 
	$$\left(M^4_k,\tilde{g}_k,\tilde{f}_k,y_k\right)\xrightarrow{pointed- \hat{C}^\infty-Cheeger-Gromov}\left(M^4_\infty,g_\infty,f_\infty,y_\infty\right),$$
	where the limit space $(M_\infty, g_\infty,f_\infty,y_\infty )$ is a smooth Ricci steady soliton orbifold with isolated singularities. In addition, $R_{g_\infty}>0$ unless $(M_\infty, g_\infty)$ is isometric to standard $\left(R^4,g_E\right)$. In this situation, the blowup assumption $\int_{B_{g_k}(x_k,2)}R^2\leq \frac{1}{k}$ forces the limit space must be flat and contain no singularity, then the convergence is smooth everywhere, but this contradicts $\left|Rm\right|_{\tilde{g}_k}(y_k)\geq \delta$. Hence the proof is done. 
\end{proof}

Now we are ready to give the proof of Theorem \ref{t5.4}.
\begin{proof}[\textbf{Proof of Theorem \ref{t5.4}}]
	By scaling invariance of $\int_{B(x,r)}R^2dV$, up to rescaling $g$ by $r^{-2}$, we may assume $r=1$. Suppose the statement were false, then there exists a sequence of non-flat Ricci shrinkers $\left( M^4_k,g_k,f_k,p_k\right)$ and points $x_k\in \left( M^4_k,g_k,f_k,p_k\right)$, where $k\in 1,2,3,\cdots$, such that
	\begin{equation}\label{5.17*}
		\int_{B_{g_k}(x_k,2)}R_{g_k}^2\leq \frac{1}{k}, \quad f_k(x_k)\geq k, \quad but\quad \left|Rm\right|_{g_k}(x_k)\geq k.
	\end{equation}
	Notice that $\int_{B\left(x,r\right)}R^2$ is scaling invariant and the lower bound of volume ratio is given by Lemma \ref{l5.5}. Hence we can argue as in the proof of Theorem \ref{t4.1} to construct a sequence of parabolic rescaled Ricci flows 
	\begin{equation}\label{5.15*}
		\left(B_{\tilde{g}_k(0)}\left(y_k,\frac{L_k}{2\epsilon_k}\right),\tilde{g}_k(t),y_k\right)\xrightarrow{C^\infty-Cheeger-Gromov}\left(M^4_\infty,g_\infty(t),y_\infty\right), t\in \left[-\frac{1}{2}, \frac{1}{2}\right],
	\end{equation}
	where $y_k\in B_{g_k}(x_k,1)$. Especially, the blowup assumption $\int_{B_{g_k}(x_k,2)}R_{g_k}^2\leq \frac{1}{k}$ and $L_k\longrightarrow+\infty$ imply $R_{g_\infty(0)}\equiv0$.
	
	\textbf{Claim.} $\left(M^4_\infty,g_\infty(0),y_\infty\right)$ is flat.
	\begin{proof}[\textbf{Proof of the claim}]
		Since $\tilde{g}_k(t)$ is a rescaled flow of the Ricci shrinker flow $g_k(t)$, thus its scalar curvature $R(x,t)\geq0$ for all $x$, then $R_{g_\infty(t)}(x)\geq 0$ for all $x\in M^4_\infty$. On the other hand,
		$$\frac{\partial R_{g_\infty(t)}}{\partial t}=\Delta_{g_\infty(t)}R_{g_\infty(t)}+2\left| Ric\right|^2_{g_\infty(t)}.$$
		Now at $t=0$, $R_{g_\infty(0)}\equiv0$ implies $\Delta_{g_\infty(0)}R_{g_\infty(0)}$, hence by the maximum principle, the fact $R_{g_\infty(t)}\geq0$ forces that $\left| Ric\right|^2_{g_\infty(0)}\equiv0$, i.e. $\left(M^4_\infty,g_\infty(0),y_\infty\right)$ is Ricci-flat.
		
		For simplicity, in the following, we denote $\tilde{g}_k(0)$ by $\tilde{g}_k$ and $g_\infty(0)$ by $g_\infty$. By our constructions in the proof of Theorem \ref{t4.1} and the notations therein, $\tilde{g}_k=T_kg_k$ where $T_k=\varepsilon\left| Rm\right|_{g_k}(x_k)\left|Rm\right|_{\hat{g}_k}(y_k)\longrightarrow+\infty$ and $\varepsilon$ is some uniform constant. 
		
		In this situation, let $\tilde{f}_k=f_k-f_k(y_k)$, then direct computations give (cf. \cite[section 3]{LW3})
		$$Ric_{\tilde{g}_k}+Hess_{\tilde{g}_k}\tilde{f}_k=\frac{\tilde{g}_k}{2T_k}.$$
		Resulting from the picks of $y_k$, we have $\left|Rm\right|_{\tilde{g}}\leq 2$ on $B_{\tilde{g}}\left(y_k,\frac{L_k}{2\varepsilon_k(n)}\right)$ for all $k$ large enough and $L_k\longrightarrow+\infty$, hence $\left|Hess_{\tilde{g}_k}\tilde{f}_k\right|_{\tilde{g}}\leq C$.
		Moreover, we have
		\begin{align}\label{5.19'}
			\left|\nabla_{\tilde{g}_k}\tilde{f}_k\right|_{\tilde{g}_k}=\frac{\left|\nabla_{g_k}f_k\right|_{g_k}}{\sqrt{T_k}},\quad \left|Hess_{\tilde{g}_k}\tilde{f}_k\right|_{\tilde{g}_k}=\frac{\left|Hess_{g_k}f_k\right|_{g_k}}{T_k}.
		\end{align}
		Notice that by Lemma \ref{l5.8} with $r=1$ and the assumption (\ref{5.17*}), if we choose a uniform $\delta>0$ small enough such that $R_{g_k}\leq f_k/2$ on $B_{g_k}(y_k,1)$, then $\left|\nabla_{\tilde{g}_k}\tilde{f}_k\right|_{\tilde{g}_k}(y_k)>0$ for all $k$  large enough. Next we consider the function
		$$\check{f}=\frac{\tilde{f}_k}{\left|\nabla_{\tilde{g}_k}\tilde{f}_k\right|_{\tilde{g}_k}(y_k)}.$$
		 By (\ref{5.19'}) and the soliton equations (\ref{1.1}), (\ref{1.3}),
		\begin{equation*}
			\left|Hess_{\tilde{g}_k}\check{f}_k\right|_{\tilde{g}_k}=\frac{\left|Hess_{g_k}f_k\right|_{g_k}}{\sqrt{T_k}\left|\nabla_{g_k}f_k\right|_{g_k}(y_k)}=\left(\frac{1-R_{g_k}+\left|Ric_{g_k}\right|_{g_k}^2}{T_k\left(f_k-R_{g_k}\right)(y_k)}\right)^{\frac{1}{2}}.
		\end{equation*}
	As a consequence, by Lemma \ref{l5.8}, the fact that $\left|Ric_{g_k}\right|_{g_k}^2\leq C \left|Rm_{g_k}\right|_{g_k}$ and $T_k\longrightarrow+\infty$, we arrive at
		$$\left|Hess_{\tilde{g}_k}\check{f}_k\right|_{\tilde{g}_k}\longrightarrow0.$$
		In summary, for all $k$ large enough, 
		$$\check{f}_k(y_k)=0,\quad \left|\nabla_{\tilde{g}_k}\check{f}_k\right|_{\tilde{g}_k}(y_k)=1,\quad \left|Hess_{\tilde{g}_k}\check{f}_k\right|_{\tilde{g}_k}\longrightarrow0\quad on\quad B_{\tilde{g}}\left(y_k,\frac{L_k}{2\varepsilon_k(n)}\right).$$
		Accordingly, under the smooth convergence of (\ref{5.15*}), we can find a nontrivial smooth limit function $f_\infty$ on $\left( M^4_\infty, g_\infty\right)$ such that it is 
		linear, hence $(M^4_\infty,g_\infty)$ splits isometrically as $\mathbb{R}\times N^3$ where $N^3$ is Ricci flat. However, it is well-known that the sectional curvature of any 3 dimensional Einstein manifold is constant, therefore $\left(M^n_\infty,g_\infty,f_\infty,y_\infty\right)$ must be flat.
	\end{proof}
	
	According to our choice of $y_k$ which is the same as in the proof of Theorem \ref{t4.1}, for some uniform $\epsilon$, we have $$\left|Rm\right|_{\tilde{g}_k}(y_k)=2\epsilon_k^2\left|Rm\right|_{\bar{g}_k}(y_k)=2\epsilon_k^2=2\epsilon^2>0,$$
	which contradicts the fact that $\left(M^4_\infty,g_\infty,f_\infty,y_\infty\right)$ is flat. Then we complete the proof the proof of Theorem \ref{t5.4}.
\end{proof}

By the way, according to the same arguments as in Theorem \ref{t4.4}, the Harnark inequality for scalar curvature also holds true under the condition of Theorem \ref{t5.4}. In summary, we have
\begin{corollary}\label{c5.9}
	For any $x\in(M^4,g,f,p)$, there exists a uniform $\varepsilon>0$ and such that if 
	\begin{equation*}
		\int_{B\left(x,3\right)}R^2\leq\varepsilon,
	\end{equation*}
	then for any $y\in B\left(x,1\right)$,
	\begin{equation*}
		\frac{R(y)}{R(x)}\leq\frac{1}{\varepsilon}.
	\end{equation*}
\end{corollary}

Recall that $0\geq\boldsymbol{\nu}\left(B(x,r),r^2\right)$$\longrightarrow0$ as $r\longrightarrow0$(cf. \cite[section 2]{Wb}). Generally, it's hard to find a constant $r_0>0$ such that $\boldsymbol{\nu}\left(B(x,r_0),r_0^2\right)\geq -\varepsilon_0$ for some given small $\varepsilon_0>0$. However, the smallness of $\boldsymbol{\nu}\left(B(x,r),r^2\right)$ is able to reveal the local geometric properties of a manifold such as the boundedness of curvature in our Theorem \ref{t1.5*} and the pseudo-locality theorem of \cite[Theorem 24]{LW}. Moreover, once we know the curvature is bounded by Theorem \ref{t1.5*}, then by the equivalent relations of \cite[Theorem 5.9]{WB2}, we can further control the local isoperimetric constant and volume ratio. In a word, it is significant to control $\boldsymbol{\nu}\left(B(x,r),r^2\right)$ for a given point $x$. Now combining the arguments as to prove Theorem \ref{t4.1} and \ref{t5.4}, we have
\begin{theorem}\label{t5.10*}
	For any $x\in(M^4,g,f,p)$, $\delta>0$ and $r\in\left(0,1\right]$, there exists a  constant $\varepsilon=\varepsilon(\boldsymbol{\mu},\delta)>0$ such that if
	\begin{equation*}
		\int_{B\left(x,\varepsilon^{-1}r\right)}R^{2}dV\leq\varepsilon,\quad f(x)\geq \varepsilon^{-1}r^{-2},
	\end{equation*}
	then
	\begin{equation*}
		\boldsymbol{\nu}\left(B\left(x,r\right) ,r^2\right)\geq -\delta.
	\end{equation*}
\end{theorem}
\begin{proof}
	The proof is almost the same as that of Theorem \ref{t5.4} but much more simple. As before, we may assume $r=1$. By Theorem \ref{t5.4}, for $\varepsilon$ small enough, $\left|Rm\right|$ is bounded on $B\left(x,\frac{\varepsilon^{-1}}{2}\right)$, then we can carry on the same blowup arguments as in the proof of Theorem \ref{t5.4} and then deduce the limit space must be flat. Therefore we complete the proof.
\end{proof}

\section{Structure at infinity}
In this section, we turn to the study of local geometry at infinity. First we recall some corresponding definitions. A cone is a manifold $\left[0,+\infty\right)\times N$ endowed with a Riemannian metric $g_c=dr^2+r^2g_N$, where $\left(N,g_N\right)$ is a closed $n-1$ dimensional Riemannian manifold. Denote $N_S=(S,+\infty)\times N$ for $S\geq0$ and define the dilation by $\lambda$ to be the map $\rho_\lambda: N_0\rightarrow N_0$ given by $\rho_\lambda(s,t)=(\lambda s,t)$. A Riemannian manifold $(M^n,g)$ is said to be $C^k$ asymptotic to the cone $\left(N_0,g_c\right)$ if for some $S>0$, there is a diffeomorphism $\varphi:N_S\rightarrow M^n\setminus \Omega$ such that $\lambda^{-2}\rho_\lambda^*\varphi^*g\longrightarrow g_c$ in $C^k_{loc}\left(N_0,g_c\right)$, where $\Omega$ is a compact subset of $M^n$.

In this direction, Kotschwar-Wang \cite{KW} showed that a shrinking Ricci soliton with quadratic curvature decay must be asymptotic to a cone and any two shrinking Ricci solitons $C^2$ asymptotic to the same cone must be isometric, and later, Munteanu-Wang \cite{MW2} proved
\begin{theorem}\label{t5.1}
	Let $(M^n, g, f,p)$ be a Ricci shrinker of dimension $n$. Then
	
	(1) If $\left|Ric\right|\longrightarrow0$ at infinity, then it is $C^k$ asymptotic to a cone for all $k$.
	
	(2) If its curvature is bounded and assume the scalar curvature converges to zero at infinity. Then $(M^n, g, f,p)$ is $C^k$ asymptotic to a cone for all $k$.
\end{theorem}
After Munteanu-Wang, Chow-Lu \cite{CL} showed that the asymptotic cone of Theorem \ref{t5.1}(1) is unique up to an isometry.

In order to derive (2) of the above theorem, Munteanu-Wang proved if $(M^n, g, f,p)$ has bounded curvature, then $\left|Ric\right|\longrightarrow 0$ as $R\longrightarrow 0$ at infinity.
In this section, we are going to give more specific descriptions about how to control Ricci curvature by scalar curvature when the curvature is bounded, and then give some natural conditions which force the Ricci shrinker asymptotic to a cone. Above all, the $\varepsilon$-regularity theorem \ref{t5.4} plays the key role in this section.

\begin{lemma}\label{l5.2}
	For any non-flat $(M^n,g,f,p)$ with $\left|Rm\right|\leq K$ for some constant $K>0$, there exists a constant $C=C(n,K)>0$ such that
	\begin{equation}\label{5.1}
		\frac{\left| Ric\right|^2}{R}\leq C\quad on\quad(M^n,g,f,p).
	\end{equation}
\end{lemma}
\begin{proof}
	As usual, we define $\Delta_f(\cdot)=\Delta-\left\langle\nabla\cdot,\nabla f \right\rangle $, then by (3.25) of \cite{GJ},
	\begin{equation}\label{5.2}
		\Delta_f\left(\left|Ric\right|^2\right)\geq 2\left|\nabla  Ric\right|^2+2\left| Ric\right|^2-C_1(n)\left|Rm\right|\left| Ric\right|^2,
	\end{equation}
	where $C_1(n)$ is a dimensional constant.
	
	Combining (\ref{1.5}) and (\ref{5.2}), we directly compute
	\begin{align}
		\Delta_f\left(\frac{\left|Ric\right|^2}{R}\right)&=\frac{\Delta_f\left(\left|Ric\right|^2\right)}{R}+\left|Ric\right|^2\Delta_f\left(\frac{1}{R}\right)+2\left\langle \nabla\left|Ric\right|^2,\nabla\left(\frac{1}{R}\right)\right\rangle\nonumber\\
		&\geq\frac{1}{R}\left( 2\left|\nabla  Ric\right|^2+2\left| Ric\right|^2-C_1(n)\left|Rm\right|\left|Ric\right|^2\right) \nonumber\\
		&\quad+\left|Ric\right|^2\left(\frac{2\left|Ric\right|^2-R}{R^2}
		+\frac{2\left|\nabla R\right|^2}{R^3}\right)-2\left\langle \nabla\left|Ric\right|^2,\frac{\nabla R}{R^2}\right\rangle.\label{5.3}
	\end{align}
	Since
	\begin{equation*}
		\frac{2\left|\nabla  Ric\right|^2}{R}+\frac{2\left|Ric\right|^2\left|\nabla R\right|^2}{R^3}\geq\frac{2\left\langle \nabla\left|Ric\right|^2,\nabla R \right\rangle }{R^2},
	\end{equation*}
	then (\ref{5.3}) yields
	\begin{align}\label{5.4}
		\Delta_f\left(\frac{\left|Ric\right|^2}{R}\right)\geq\frac{1}{R}\left(2\left|Ric\right|^2-C_1(n)\left|Rm\right|\left|Ric\right|^2\right)+\left|Ric\right|^2\left(\frac{2\left|Ric\right|^2-R}{R^2}\right).
	\end{align}
	Setting $\lambda=\frac{\left|Ric\right|^2}{R}$ and $C_2=C_2(n)=-\min\left\lbrace 1-C_1(n)\left|Rm\right|,0\right\rbrace\geq 0$, then we can rewrite (\ref{5.4}) as
	\begin{align}\label{5.5}
		\Delta_f\lambda&
		\geq 2\lambda^2+\left(1-C_1(n)\left|Rm\right|\right)\lambda\nonumber\\
		&\geq2\lambda^2-C_2\lambda.
	\end{align}
	
	If $(M^n,g,f,p)$ is compact, then (\ref{5.1}) follows from the maximum principle directly. If $(M^n,g,f,p)$ is non-compact, we need some more discussions. Let $\phi(x)\in C^\infty_c\left[0,+\infty \right)$ be a smooth cut-off function with compact support and satisfies
	$$0\leq\phi\leq1, \quad\hspace*{0.3em}\phi(t)\equiv1\hspace*{0.3em} \mbox{for}\hspace*{0.3em}0\leq t\leq r,\quad\hspace*{0.3em}\phi(t)\equiv0\hspace*{0.3em}\mbox{for}\hspace*{0.3em}t\geq2r\hspace*{0.3em}\quad\mbox{and}\quad\hspace*{0.3em}\phi'\leq0.$$
	Moreover, we can also require $\left|\phi'\right|\leq\frac{2}{r}$ as well as $\left|\phi''\right|\leq\frac{4}{r^2}$. Since $(M^n,g,f,p)$ is non-compact, we may assume $r\geq 1$. Now we use $\phi\left(f(x)\right)$ as a cut-off function on $D(2r)$ and then compute
	\begin{equation}\label{5.6}
		\left|\nabla\phi\right|\leq\frac{C_3}{\sqrt{r}},\quad \left|\Delta_f\phi\right|\leq C_3,
	\end{equation}
	where $C_3=C_3(n)$ is a constant.
	
	Direct computations with (\ref{5.5}) and (\ref{5.6}) imply that the test function $G=\phi^2\lambda$ satisfies
	\begin{align*}
		\phi^2\Delta_fG&=\phi^2\Delta_f\left(\phi^2\lambda\right)\nonumber\\
		&=\phi^4\Delta_f\lambda+\phi^2\lambda\Delta_f\phi^2+2\phi^2\left\langle\nabla \lambda, \nabla\phi^2\right\rangle\nonumber\\
		&\geq2G^2-C_2G+G\Delta_f\phi^2+2\left(\left\langle \nabla G, \nabla\phi^2\right\rangle-\lambda\left\langle\nabla\phi^2,\nabla\phi^2\right\rangle\right)\nonumber\\
		&\geq2G^2-C_4G+2\left\langle \nabla G, \nabla\phi^2\right\rangle,
	\end{align*}
	where $C_4=C_4(C_2, C_3)>0$ is a constant.  Since $r\geq1$ is arbitrary, then the maximum principle gives the desired estimate (\ref{5.1}).
\end{proof}

By our Lemma \ref{l5.2} and Corollary 1.1 of \cite{MW3}, we immediately conclude
\begin{corollary}\label{c5.3}
	For $(M^n,g,f,p)$ with $\left|Rm\right|\leq K$ for some constant $K>0$, there exists a constant $C=C(n,K)>0$ such that if
	\begin{equation*}
		R\leq C,
	\end{equation*}
	then $(M^n,g,f,p)$ is isometric to the standard Gaussian soliton $\left(\mathbb{R}^n, dx^2, \frac{1}{4}\left|x\right|^2\right)$.
\end{corollary}
Actually,  (\ref{5.1}) of Lemma \ref{l5.2} can be improved to a local estimate, and the method is almost the same as to prove the Harnark inequality of Theorem \ref{t4.4}. The only difference is that we need the $\boldsymbol{\mu}$ functional to control the lower bound of volume ratio(cf. \cite[Theorem 23]{LW}).
\begin{proposition}
For any $x\in(M^n,g,f,p)$, $r>0$ and $K>0$, there exists a constant $\varepsilon=\varepsilon(n,K,\boldsymbol{\mu})>0$ such that if
\begin{equation*}
		\left|Rm\right|(x)\leq Kr^{-2}\quad on \quad B\left(x,2r\right),
\end{equation*}
then
\begin{equation*}
	\frac{\left|Rm\right|(y)}{R(x)}\leq \varepsilon^{-1}\quad for \quad y\in B\left(x,r\right).
\end{equation*}
\end{proposition}

Now we are ready to prove the main theorem of this section.
\begin{theorem}[\textbf{=Theorem \ref{t1.8*}}]\label{t5.7}
For any non-compact Ricci shrinker $(M^4,g,f,p)$, if for some $q>2$,
\begin{equation}\label{5.17'}
\int_{M^4}R^qdV<+\infty,
\end{equation}
then it is $C^k$ asymptotic to a cone for all $k$.
\end{theorem}
\begin{proof}
On any $B(x,r)$, 
$$\int_{B(x,r)}R^2\leq \left( \int_{B(x,r)}R^q\right) ^{\frac{1}{q}}\left|B\left(x,r\right)\right|^{\frac{q-1}{q}},$$
then by the non-expanding estimate of volume ratio in \cite[Theorem 1.2]{LW2},
\begin{equation}\label{5.18''}
\int_{B(x,r)}R^2\leq C(q)r^{\frac{4(q-1)}{q}}\left( \int_{B(x,r)}R^q\right)^{\frac{1}{q}}.
\end{equation}
In view of (\ref{5.17'}) and (\ref{5.18''}), outside some sufficiently large geodesic ball $B(p,r_0)$, 
\begin{equation*}
	\int_{B\left(x,10\right)}R^2=\delta\leq\varepsilon,\quad 	\delta\longrightarrow 0\quad as\hspace*{0.5em} x\longrightarrow \infty,
\end{equation*}
where $\varepsilon$ is the same as in Corollary \ref{c5.9} and $B\left(x,10\right)$ is an arbitrary geodesic ball contained in $M^4\setminus B(p,r_0)$. Therefore by Theorem \ref{t5.4}, there exists a uniform constant $K>0$ such that for all $x\in M^4$,
\begin{equation}\label{5.19}
\left|Rm\right|(x)\leq K.
\end{equation}
Next let $y\in B\left(x,1\right)$ be an arbitrary point, then $B\left(y,2d(x,y) \right)\subset B\left(x,10\right)$ and $x\in B\left(y,2d(x,y)\right)$. Thus by Corollary \ref{c5.9}, for some $c>0$,
\begin{equation*}
cR(y)\geq R(x).
\end{equation*}
In this situation, similar to (\ref{5.18''}), we have
\begin{align*}\label{5.20}
\left|B\left(x,1\right)\right|R(x)&\leq c\int_{B\left(x,1\right)}R(y)dV(y)\\
&\leq c\left( \int_{B\left(x,1\right)} R^2\right)^{\frac{1}{2}}\left|B\left(x,1\right)\right|^{\frac{1}{2}}\\
&\leq C\delta^{\frac{1}{2}}.
\end{align*}
By Lemma \ref{l5.5}, $\left|B\left(x,1\right)\right|\geq C(\boldsymbol{\mu})$, finally we arrive at 
\begin{equation*}
R(x)\leq C\delta^{\frac{1}{2}}.
\end{equation*}
In other words, $R(x)\longrightarrow 0$ as $x\longrightarrow \infty$ uniformly. Now by (\ref{5.19}), Theorem \ref{t5.1} and Lemma \ref{l5.2}, we accomplish the proof.
\end{proof}

As a corollary, we provide a different proof of the following result which was proved by Munteanu-Wang \cite[Theorem 5.1]{MW1*}.
\begin{corollary}
For any Ricci shrinker $(M^4,g,f,p)$, if $R\longrightarrow0$ at infinity, then it is $C^k$ asymptotic to a cone for all $k$.
\end{corollary}
\begin{proof}
For any small $\delta>0$, there exist a sufficiently large constant $r_0>0$ such that for all $x\in M^4\setminus B(p,r_0)$,
$$R(x)\leq\delta.$$
Let $\varepsilon$ be the same as in Theorem \ref{t5.4}, by the non-expanding estimate of volume ratio in \cite[Theorem 1.2]{LW2}, 
$$\left|B\left(x,1\right)\right|\leq C.$$
Therefore, for any $x\in M^4\setminus B(p,2r_0)$,
$$\int_{B\left(x,1\right)}R^2(y)dV(y)\leq C\delta^2.$$ 
Since $\delta\longrightarrow 0$ as $x\longrightarrow \infty$ uniformly, then by Theorem \ref{t5.1}, \ref{t5.4} and Lemma \ref{l5.2}, the proof is completed.
\end{proof}

For general dimension $n\geq 4$, we have
\begin{theorem}\label{t5.10}
	For any $x\in(M^n,g,f,p)$ and $r>0$, there exists a constant $\varepsilon=\varepsilon\left(n, \boldsymbol{\mu}\right) >0$ such that if
\begin{equation}\label{5.25}
	\int_{B\left(x,r\right)}\left|Rm\right|^{\frac{n}{2}}dV\leq\varepsilon,
\end{equation}
then
\begin{equation*}
	\left|Rm\right|(x)\leq \varepsilon^{-1}r^{-2}.
\end{equation*}
\end{theorem}
\begin{proof}
The proof is the same as that of Theorem \ref{t5.4} but much more simple, only noticing that (\ref{5.25}) ensures the limit manifold must be flat and in this situation, there is no need to assume the largeness of $f(x)$ .
\end{proof}

Consequently, by the Harnack inequality of Proposition \ref{p4.3}, Theorem \ref{t5.10} and the same arguments as in Theorem \ref{t5.7}, we have
\begin{corollary}
For any Ricci shrinker $(M^n,g,f,p)$, if for some $q>\frac{n}{2}$,
\begin{equation*}
	\int_{M^n}\left|Rm\right|^qdV<+\infty,
\end{equation*}
then it is $C^k$ asymptotic to a cone for all $k$.
\end{corollary}

\noindent {\it\textbf{Acknowledgements}}: The author Y. Wang is supported partially by NSFC (Grant No.11971400) and National key Research and Development projects of China (Grant No. 2020YFA0712500). The authors are grateful to Professor Bing Wang for many constructive suggestions.

\bibliographystyle{amsalpha}

\begin{thebibliography}{2}
\bibitem{A1}  M. Anderson, \emph{Ricci curvature bounds and Einstein metrics on compact manifolds}, J. Amer. Math. Soc. {\bf 2}(1989), 455-490.
	
\bibitem{A2}  M. Anderson, \emph{Convergence and rigidity of manifolds under Ricci curvature bounds}, Invent. Math. {\bf 102}(1990), no.2, 429-445.
	
\bibitem{CCG}  Chow, B., Chu, S.-C., Glickenstein, D., Guenther, C., Isenberg, J., Ivey, T., Knopf, D., Lu, P., Luo, F., Ni, L, \emph{The Ricci Flow: Techniques and Applications. Part I. Geometric Aspects}, Mathematical Surveys and
Monographs, vol 135. American Mathematical Society, Providence (2007).	

\bibitem{CCG3}  Chow, B., Chu, S.-C., Glickenstein, D., Guenther, C., Isenberg, J., Ivey, T., Knopf, D., Lu, P., Luo, F., Ni, L, \emph{The Ricci Flow: Techniques and Applications. Part III. Geometric-Analytic Aspects}, Mathematical Surveys and
Monographs, vol 163. American Mathematical Society, Providence (2010).	

\bibitem{CL} B. Chow and P. Lu, \emph{Uniqueness of asymptotic cones of complete noncompact shrinking gradient Ricci solitons with Ricci curvature decay}, C. R. Acad. Sci. Paris, Ser. I. {\bf 353}(2015), 1007-1009.	

\bibitem{CLY1} B. Chow, P. Lu and B. Yang, \emph{Lower bounds for the scalar curvatures of noncompact gradient Ricci solitons}, C. R. Acad. Sci. Paris, Ser. I. {\bf 349}(2011), 1265-1267.	
	
\bibitem{CLY2}  B. Chow,  P. Lu and B. Yang, \emph{A necessary and sufficient condition for Ricci shrinkers to have positive AVR}, P. Am. Math. Soc. {\bf 140}(2011), no.6, 2179-2181.

\bibitem{CZ}  H.-D. Cao and D.-T. Zhou, \emph{On complete gradient shrinking Ricci solitons}, J. Differ. Geom. {\bf 85}(2010), no.2, 175-186.

\bibitem{CGT}  J. Cheeger, M. Gromov and M. Taylor, \emph{Finite propagation speed, kernel estimates for functions of the Laplace operator, and the geometry of complete Riemannian manifolds}, J. Differ. Geom. {\bf 17}(1982), no.1, 15-53.

\bibitem{EMT}  J. Enders, R. M\"{u}ller and P. Topping, \emph{On type-\uppercase\expandafter{\romannumeral 1} singularities in Ricci flow}, Comm. Anal. Geom. {\bf 19}(2011) 905-922.

\bibitem{FIK}  M. Feldman, T. Ilmanen, and D. Knopf, \emph{Rotationally symmetric shrinking and expanding gradient k\"{a}hler–Ricci solitons}, J. Differ. Geom. {\bf 65}(2003), 169–209.

\bibitem{FS}  A. Futaki and Y. Sano, \emph{Lower diameter bounds for compact shrinking Ricci solitons}, Asian. J. Math. {\bf 17}(2013), no.1, 17-32.	

\bibitem{GJ}  H. B. Ge and W. S. Jiang, \emph{$\epsilon$-regularity for shrinking Ricci solitons and Ricci flows}, Geom. Funct. Anal. {\bf 27}(2017), 1231-1256.	

\bibitem{HM}  R. Haslhofer and R. M\"{u}ller, \emph{A compactness theorem for complete Ricci shrinkers}, Geom. Funct. Anal. {\bf 21}(2011), 1091-1116.

\bibitem{HM2}  R. Haslhofer and R. M\"{u}ller, \emph{A note on the compactness theorem for 4d Ricci shrinkers}, Proc. Amer. Math. Soc. {\bf 143}(2015), no. 10, 4433-4437.

\bibitem{HS}  S. Huang, \emph{$\varepsilon-$ regularity and structure of four-dimensional shrinking Ricci solitons}, IMRN, (2020), no.5, 1511-1574.

\bibitem{KW}  B. Kotschwar and L. Wang, \emph{Rigidity of asymptotically conical shrinking gradient Ricci solitons}, J. Differ. Geom. {\bf 100}(2015), 55-108.

\bibitem{LLW}  H. Li, Y. Li and B. Wang, \emph{On the structure of Ricci shrinkers}, J. Funct. Anal., {\bf 280}(2021), 108955.

\bibitem{LW}  Y. Li and B. Wang, \emph{Heat kernel on Ricci shrinkers }, Calc. Var. Partial Differential Equations, {\bf 59}(2020), no.6, 194.

\bibitem{LW2}  Y. Li and B. Wang, \emph{Heat kernel on Ricci shrinkers (\uppercase\expandafter{\romannumeral2}) }, arXiv:2301.08430.

\bibitem{LW3}  Y. Li and B. Wang, \emph{On K\"{a}hler Ricci shrinker surfaces}, arXiv:2301.09784.

\bibitem{MW1}  O. Munteanu and J. P. Wang, \emph{Analysis of weighted Laplacian and applications to Ricci solitons}, Commun. Anal. Geom. {\bf 20}(2012), no.1, 55–94.

\bibitem{MW1*}  O. Munteanu and J. P. Wang, \emph{Geometry of shrinking Ricci solitons}, Compos. Math. {\bf 151}(2015), no.12, 2273-2300.

\bibitem{MW2}  O. Munteanu and J. P. Wang, \emph{Conical structure for shrinking Ricci solitons}, J. Eur. Math. Soc. {\bf 19}(2017), 3377-3390.

\bibitem{MW3}  O. Munteanu and M.-T. Wang, \emph{The curvature of gradient Ricci solitons}, Math. Res. Lett. {\bf 18}(2011), no.06, 1051-1069.

\bibitem{Shi} W. X. Shi, \emph{Deforming the metric on complete Riemannian manifolds}, J. Differ. Geom. {\bf 30}(1989), no.2, 223–301.

\bibitem{Wb} B. Wang, \emph{The local entropy along Ricci flow Part A: the no-local-collapsing theorems}, Camb. J. Math. {\bf 6}(2018), no. 3, 267-346.

\bibitem{WB2} B. Wang, \emph{The local entropy along Ricci flow Part B: the pseudo-locality theorems}, arxiv:2010.09981v1.
\end{thebibliography}

Jie Wang,  Institute of Geometry and Physics, University of Science and Technology of China, No. 96 Jinzhai Road, Hefei, Anhui Province, 230026, China.

Email: wangjie9math@163.com

Youde Wang, 1. School of Mathematics and Information Sciences, Guangzhou University; 2. Hua Loo-Keng Key Laboratory of Mathematics, Institute of Mathematics, Academy of Mathematics and Systems Science, Chinese Academy of Sciences, Beijing 100190, China; 3. School of Mathematical Sciences, University of Chinese Academy of Sciences, Beijing 100049, China.

Email: wyd@math.ac.cn

\end{document}